\documentclass[12pt,letterpaper,showkeys]{article}
\usepackage{graphicx,psfrag}
\usepackage[captionskip=1ex]{subfig}
\usepackage{amssymb}
\usepackage{amsmath}
\usepackage{mathrsfs}
\newtheorem{theorem}{Theorem}[section]
\newtheorem{lemma}[theorem]{Lemma}

\newtheorem{remark}[theorem]{Remark}

\newtheorem{assumption}{Assumption}

\newcommand{\ba}{\begin{array}}
\newcommand{\ea}{\end{array}}
\newcommand{\beq}{\begin{equation}}
\newcommand{\eeq}{\end{equation}}
\newcommand{\beqa}{\begin{eqnarray}}
\newcommand{\eeqa}{\end{eqnarray}}
\newcommand{\beqas}{\begin{eqnarray*}}
\newcommand{\eeqas}{\end{eqnarray*}}
\newcommand{\bi}{\begin{itemize}}
\newcommand{\ei}{\end{itemize}}
\newcommand{\gap}{\hspace*{2em}}

\newcommand{\nn}{\nonumber}

\def\eqref#1{(\ref{#1})}

\def\QED{\ifhmode\unskip\nobreak\fi\ifmmode\ifinner\else\hskip5pt\fi\fi
  \hbox{\hskip5pt\vrule width5pt height5pt depth1.5pt\hskip1pt}}
\def\aalpha{{\theta}}

\def\bd{{\bar d}}
\def\bDelta{{\bar\Delta}}
\def\bE{{\bf E}}
\def\bg{{\bar g}}
\def\bj{{\bar j}}
\def\bx{{\bar x}}
\def\cK{{\cal K}}
\def\cS{{\cal S}}
\def\dist{{\rm dist}}
\def\dom{{\rm dom}}
\def\eps{{\epsilon}}
\def\P{{\bf P}}
\def\q{{\Psi}}

\def\hg{{\hat g}}
\def\td{{\tilde d}}
\def\tF{{\tilde F}}
\def\tx{{\bar x}}

\def\Bi{{i}}
\def\td{{\tilde d}}

\title{ A Randomized Nonmonotone Block Proximal Gradient Method for a Class of Structured 
Nonlinear Programming}

\author{
	Zhaosong Lu%
	\thanks{
	Department of Mathematics, Simon Fraser University, Burnaby, BC,
	V5A 1S6, Canada. (email: {\tt zhaosong@sfu.ca}). This author was supported
        in part by NSERC Discovery Grant.}
	\and
	Lin Xiao
	\thanks{Machine Learning Groups, Microsoft Research, One Microsoft Way, Redmond, WA 98052,   
         USA. (email: {\tt lin.xiao@microsoft.com}).}
}

\begin{document}

\maketitle

\begin{abstract}

We propose a randomized nonmonotone block proximal gradient (RNBPG) method
for minimizing the sum of a smooth (possibly nonconvex) function and a block-separable 
(possibly nonconvex nonsmooth)  function. At each iteration, this method randomly picks a block 
according to any prescribed probability distribution and solves typically several associated proximal 
subproblems that usually have a closed-form solution, until a certain progress on objective 
value is achieved. In contrast to the usual randomized block coordinate descent method \cite{RichtarikTakac12,PaNe13}, our method 
has a nonmonotone flavor and uses variable stepsizes that can partially utilize the local 
curvature information of the smooth component of objective function.  
We show that 
any accumulation point of the solution sequence of the method is a stationary point of the problem 
{\it almost surely} and the method is capable of finding an approximate stationary point with high 
probability. We also establish a sublinear rate of convergence for the method in terms of 
the minimal expected squared norm of certain proximal gradients over the iterations.  When the 
 problem under consideration is convex, we show that the expected objective values generated 
by RNBPG converge to the optimal value of the problem. Under some assumptions, 
we further establish a sublinear and linear rate of convergence on the expected objective 
values generated by a monotone version of RNBPG. Finally, we conduct some preliminary experiments 
to test the performance of RNBPG on the $\ell_1$-regularized least-squares problem and 
a dual SVM problem in machine learning. The computational results demonstrate that our 
method substantially outperforms the randomized block coordinate {\it descent} method with 
fixed or variable stepsizes.

\bigskip


\bigskip
\bigskip


\end{abstract}

\section{Introduction} \label{intro}

Nowadays first-order (namely, gradient-type) methods are the prevalent tools for solving large-scale
problems arising in science and engineering. As the size of problems becomes huge, it is, however,
greatly challenging to these methods because gradient evaluation can be prohibitively expensive. 
Due to this reason, block coordinate descent (BCD) methods and their variants have been studied for 
solving various large-scale problems (see, for example,  \cite{ChHsLi08,HsChLiKeSu08,WuLa08,LiOs09,TsYu09-1,TsYu09-2,WeGoSc,Wright10,QiScGo10,YuTo11,
ShTe09,LeLe10,RiTa11,ShZh12,RichtarikTakac12-1,TaRiGo13}). Recently, Nesterov 
\cite{Nesterov12rcdm} proposed a randomized BCD (RBCD) method, which
is promising for solving a class of huge-scale convex optimization problems, provided the involved 
partial gradients can be efficiently updated. The iteration complexity for finding an approximate 
optimal solution is analyzed in \cite{Nesterov12rcdm}. More recently, Richt\'{a}rik and Tak\'{a}\v{c}
\cite{RichtarikTakac12} extended Nesterov's RBCD method \cite{Nesterov12rcdm} to solve a more 
general class of convex optimization problems in the form of
\beq \label{NLP}
\min\limits_{x \in \Re^N} \left\{F(x) := f(x) + \q(x)\right\},
\eeq
where $f$ is convex differentiable in $\Re^N$ and $\Psi$ is a block separable convex function.
More specifically,
\[
\Psi(x) = \sum_{i=1}^n \Psi_i(x_i) ,
\]
where each $x_i$ denotes a subvector of~$x$ with cardinality~$N_i$, $\{x_i :i=1,\ldots,n\}$ form a
partition of the components of~$x$, and each~$\Psi_i: \Re^{N_i} \to \Re\cup\{+\infty\}$ is a closed
convex function.

Given a current iterate $x^k$, the RBCD method \cite{RichtarikTakac12} picks
$i \in\{1,\ldots,n \}$ uniformly, solves a block-wise proximal subproblem
in the form of
\beq \label{convex-dir}
d_i(x^k) := \arg\min_{s\in\Re^{N_i}} \left\{ \nabla_i f(x^k)^Ts
+ \frac{L_i}{2}\|s\|^2 + \Psi_i(x^k_i + s) \right\},
\eeq
and sets $x^{k+1}_i = x^k_i+d_i(x^k)$ and $x^{k+1}_j = x^{k}_j$ for all $j \neq i$, where
$\nabla_i f\in\Re^{N_i}$ is the \emph{partial gradient} of~$f$ with respect to~$x_i$ and $L_i>0$
is the Lipschitz constant of $\nabla_i f$ with respect to the norm $\|\cdot\|$ (see Assumption 
\ref{assump-2} for details). The iteration complexity of finding an approximate optimal solution
 with high probability is established in \cite{RichtarikTakac12} and has recently been improved by 
Lu and Xiao \cite{LuXiao13-1}.  Very recently, Patrascu and Necoara \cite{PaNe13} extended 
this method to solve problem \eqref{NLP} in which $F$ is nonconvex, and  they studied 
convergence of the method under the assumption that the block is chosen uniformly at 
each iteration.


One can observe that for $n=1$, the RBCD method \cite{RichtarikTakac12,PaNe13} becomes 
a classical proximal (full) gradient method with a constant stepsize $1/L$. It is known that  
the latter method  tends to be practically much slower than the same type of methods but 
with variable stepsizes, for example, spectral-type stepsize \cite{BaBo88,BiMaRa00,DaZh01,WrNoFi09,LuZh12})  that utilizes partial local curvature 
information of the smooth component $f$. The variable stepsize strategy  shall also be  
applicable to the RBCD method and improve its practical performance dramatically. In addition, 
the RBCD method is a monotone method, that is, the objective values generated by the method 
are monotonically decreasing. As mentioned in the literature (see, for example, \cite{FeLuRo96,GrLu86,
ZhHa04}), nonmonotone methods often produce solutions of better quality than the monotone 
counterparts for nonconvex optimization problems. These motivate us to propose a randomized 
nonmonotone block proximal gradient method with variable stepsizes for solving a class of
 (possibly nonconvex) structured nonlinear programming problems in the form of \eqref{NLP} 
satisfying Assumption \ref{assump-2} below.

Throughout this paper we assume that the set of optimal solutions of problem \eqref{NLP}, denoted 
by $X^*$, is nonempty and the optimal value of \eqref{NLP} is denoted by $F^\star$. For simplicity 
of presentation, we associate $\Re^N$ with the standard Euclidean norm, denoted by $\|\cdot\|$. 
We also make the following assumption.

\begin{assumption} \label{assump-2}
$f$ is differentiable (but possibly nonconvex) in $\Re^N$.  Each~$\Psi_i$ is a 
(possibly nonconvex nonsmooth) function from $\Re^{N_i}$ to $\Re\cup\{+\infty\}$ for $i=1,\ldots,n$.
The gradient of function $f$ is coordinate-wise Lipschitz continuous
with constants $L_i>0$ in $\Re^N$, that is,
\[
\|\nabla_i f(x+h) - \nabla_i f(x)\| \ \le \ L_i \|h\| \quad\quad \forall h \in \cS_i, \
i=1,\ldots,n, \quad \forall x \in \Re^N,
\]
where 
\[
\cS_i = \left\{(h_1,\ldots,h_n) \in \Re^{N_1} \times \cdots \times \Re^{N_n} : h_j =0 \quad \forall j \neq i\right\}.
\]
\end{assumption}

In this paper we propose a randomized nonmonotone block proximal gradient (RNBPG) method
for solving problem \eqref{NLP} that satisfies the above assumptions. At each iteration, this method 
randomly picks a block according to an arbitrary prescribed (not necessarily uniform) probability distribution 
and solves typically several associated proximal subproblems in the form of \eqref{convex-dir} with 
$L_i$ replaced by some $\theta$, which can be, for example, estimated by the spectral method (e.g., 
see \cite{BaBo88,BiMaRa00,DaZh01,WrNoFi09,LuZh12}),  until a certain progress on the objective value 
is achieved. In contrast to the usual RBCD method \cite{RichtarikTakac12,PaNe13}, our method 
enjoys a nonmonotone flavor and uses variable stepsizes that can partially utilize the local 
curvature information of the smooth component $f$. For arbitrary probability 
distribution\footnote{The convergence analysis  of the RBCD method conducted in 
\cite{RichtarikTakac12,PaNe13} is only for uniform probability distribution.}, We show that the 
expected objective values generated by the method converge to the expected limit of the 
objective values obtained by a random single run of the method. Moreover, any accumulation 
point of the solution sequence of the method is a stationary point of the problem 
{\it almost surely} and the method is capable of finding an approximate stationary point with high 
probability. We also establish a sublinear rate of convergence for the method in terms of 
the minimal expected squared norm of certain proximal gradients over the iterations.  When the 
 problem under consideration is convex, we show that the expected objective values generated 
by RNBPG converge to the optimal value of the problem. Under some assumptions, 
we further establish a sublinear and linear rate of convergence on the expected objective 
values generated by a monotone version of RNBPG.
Finally, we conduct some preliminary experiments to test the performance of RNBPG on the 
$\ell_1$-regularized least-squares problem and a dual SVM problem in machine learning. 
The computational results demonstrate that our method substantially outperforms the 
randomized block coordinate {\it descent} method with fixed or variable stepsizes.

This paper is organized as follows. In Section \ref{method} we propose a RNBPG method for 
solving structured nonlinear programming problem \eqref{NLP} and analyze its convergence. In 
Section \ref{monotone} we analyze the convergence of RNBPG for solving structured convex 
problem. In Section \ref{comp} we conduct numerical experiments to compare RNBPG method 
with the RBCD method with fixed or variable stepsizes.

Before ending this section we introduce some notations that are used throughout this paper and 
also state some known facts. The domain of the function $F$ is denoted by $\dom(F)$. $t^+$ 
stands for $\max\{0,t\}$ for any real number $t$. Given a closed set $S$ and a point $x$, 
$\dist(x,S)$ denotes the distance between $x$ and $S$. For symmetric matrices $X$ and $Y$, 
$X \preceq Y$ means that $Y-X$ is positive semidefinite. Given a positive definite matrix $\Theta$ 
and a vector $x$, $\|x\|_{\Theta}=\sqrt{x^T\Theta x}$.  In addition, $\|\cdot\|$ denotes the 
Euclidean norm. Finally, it immediately follows from  Assumption \ref{assump-2} that
\beq \label{lipineq}
 f(x+h) \le f(x) + \nabla f(x)^T h + \frac{L_i}{2}\|h\|^2\quad\quad \ \forall h \in \cS_i, \
i=1,\ldots,n;\ \forall x \in \Re^N.
\eeq
By Lemma 2 of Nesterov \cite{Nesterov12rcdm} and Assumption \ref{assump-2}, we also know that 
$\nabla f$ is Lipschitz continuous with constant $L_f := \sum_i L_i$, that is, 
\beq \label{g-lip}
\|\nabla f(x) -\nabla f(y)\| \le L_f\|x-y\| \quad\quad x, y \in \Re^N.
\eeq

\section{Randomized nonmonotone block proximal gradient method}
\label{method}

In this section we propose a RNBPG method for solving structured nonlinear programming 
problem \eqref{NLP} and analyze its convergence.

We start by presenting a RNBPG method as follows. At each iteration, this method 
randomly picks a block according to any prescribed (not necessarily uniform) probability distribution 
and solves typically several associated proximal subproblems in the form of \eqref{convex-dir} with 
$L_i$ replaced by some $\theta_k$ until a certain progress on objective value is achieved.

\gap

\noindent
{\bf Randomized nonmonotone block proximal  gradient (RNBPG) method} \\ [5pt]
Choose $x^0 \in \dom(F)$,  $\eta >1$, $\sigma > 0$, 
$0 < \underline\aalpha \le \bar\aalpha$,  integer
$M \ge 0$, and $0<p_i<1$ for $i=1,\ldots,n$ such that $\sum^n_{i=1} p_i=1$. Set $k=0$.
\begin{itemize}
\item[1)] Set $d^k = 0$.   Pick $i_k=i \in \{1,\ldots,n\}$ with probability $p_i$. 
Choose $\aalpha_k^0 \in [\underline\aalpha, \bar\aalpha]$.
\item[2)] For $j=0,1, \ldots$
  \bi
    \item[2a)] Let $\aalpha_k = \aalpha_k^0 \eta^j$. Compute 
\begin{equation}\label{eqn:prox-solution}
(d^k)_{i_k} = \arg\min_s \left\{\nabla_{i_k} f(x^k)^T s + \frac{\theta_k}{2}
\|s\|^2 + \q_{i_k}(x^k_{i_k}+s)\right\}.
\end{equation}
    \item[2b)] If $d^k$ satisfies
          \beq \label{reduct}
           F(x^k + d^k) \ \le \ \max_{[k-M]^+ \le i \le k} F(x^i) - \frac{\sigma}{2}  \|d^k\|^2,
          \eeq
          go to step 3).
  \ei
\item[3)] Set $x^{k+1} = x^k + d^k$, $k \leftarrow k+1$ and go to step 1).
\end{itemize}
\noindent
{\bf end}

\begin{remark}
The above method becomes a monotone method if $M=0$.
\end{remark}

\gap

Before studying convergence of RNBPG, we introduce some notations and state some facts 
that will be used subsequently.

Let $\bd^{k,i}$ denote the vector $d^k$ obtained in Step (2) of RNBPG if $i_k$ is 
chosen to be $i$. Define 
\beq \label{bdx}
\bd^k = \sum^n_{i=1} \bd^{k,i}, \quad\quad \tx^k = x^k+\bd^k.  
\eeq
One can observe that  $(\bd^{k,i})_t =0$ for $t\neq i$ 
and there exist $\theta^0_{k,i} \in [\underline\aalpha, \bar\aalpha]$ and 
 the smallest nonnegative integer $j$ such that $\theta_{k,i}=\theta^0_{k,i} \eta^j$ and
\beq
F(x^k + \bd^{k,i})  \le   F(x^{\ell(k)}) - \frac{\sigma}{2}  \|\bd^{k,i}\|^2, \label{F-reduct}
\eeq
where 
\beqa
& & (\bd^{k,i})_i =\arg\min\limits_s \left\{\nabla_{i} f(x^k)^T s + \frac{\theta_{k,i}}{2}
\|s\|^2 + \q_{i}(x^k_{i}+s)\right\}, \label{subprob-soln} \\ [4pt]
&& \ell(k) = \arg\max\limits_i \{F(x^i): i=[k-M]^+, \ldots, k\}\quad\quad \forall k \ge 0. \label{lk}
\eeqa
Let $\Theta_k$ denote the block diagonal matrix $(\theta_{k,1}I_1, \ldots,\theta_{k,n}I_n)$, where 
$I_i$ is the $N_i \times N_i$ identity matrix. By the definition of $\bd^k$ and \eqref{subprob-soln}, 
we observe that
\beq \label{bdk}
\bd^k = \arg\min_d \left\{\nabla f(x^k)^T d + \frac{1}{2}d^T \Theta_k d + \q(x^k+d)\right\}.
\eeq

After $k$ iterations, RNBPG generates a random output $(x^k,F(x^k))$, which depends 
on the observed realization of random vector
\[
\xi_k = \{i_0,\ldots,i_k\}.
\]
We define $\bE_{\xi_{-1}} [F(x^0)] = F(x^0)$. Also, define
\beqa
& & \Omega(x^0) = \{x\in\Re^N: F(x) \le F(x^0)\} \label{omega}, \\ [4pt]
& & L_{\max} = \max\limits_i L_i,\quad\quad p_{\min} = \min\limits_i p_i, \label{Lf-pmin} \\ [4pt]
& & c = \max\left\{\bar\aalpha,\eta(L_{\max}+\sigma)\right\}. \label{c}
\eeqa

\gap

The following lemma establishes some relations between the expectations of $\|d^k\|$ and 
$\|\bd^k\|$.

\begin{lemma} \label{exp-dk}
Let $d^k$ be generated by RNBPG and $\bd^k$ defined in \eqref{bdx}. There hold
\beqa
 &&   \bE_{\xi_k}[\|d^k\|^2] \ge  p_{\min} \ \bE_{\xi_{k-1}}[\|\bd^k\|^2], \label{bd-sqr} \\ [5pt]
&& \bE_{\xi_k}[\|d^k\|] \ge p_{\min} \ \bE_{\xi_{k-1}}[\|\bd^k\|]. \label{bd-norm}  
\eeqa
\end{lemma}

\begin{proof}
By \eqref{Lf-pmin} and the definitions of $d^k$ and $\bd^k$, we can observe that
\[
\ba{l}
\bE_{i_k}[\|d^k\|^2]   = \sum\limits_i p_i \|\bd^{k,i}\|^2 \ge \left(\min_i p_i\right) 
\sum\limits_i \|\bd^{k,i}\|^2 = p_{\min} \|\bd^k\|^2, \\ [10pt]
\bE_{i_k}[\|d^k\|]  = \sum\limits_i p_i \|\bd^{k,i}\| \ge \left(\min_i p_i\right)\sum\limits_i \|\bd^{k,i}\| \ge p_{\min}\sqrt{\sum\limits_i \|\bd^{k,i}\|^2} \ge p_{\min} \|\bd^k\|. 
\ea
\]
The conclusion of this lemma follows by taking expectation with respect to $\xi_{k-1}$ on both sides of the above inequalities. 
\end{proof}

\gap

We next show that the inner loops of the above RNBPG method must terminate finitely. As 
a byproduct, we provide a uniform upper bound on $\Theta_k$. 

\begin{lemma} \label{approx-lip}
Let $\{\theta_k\}$ be the sequence generated by RNBPG, $\Theta_k$ defined above, and $c$ 
defined in \eqref{c}. There hold
\bi
\item[(i)]
$\underline\aalpha \le \theta_k  \le c \quad\quad \forall k$.
\item[(ii)] $\underline\aalpha I  \preceq  \Theta_k \preceq  c I \quad\quad \forall k$.
\ei
\end{lemma}

\begin{proof}
(i) It is clear that $\theta_k \ge \underline\aalpha$. We now show $\theta_k \le c$ by 
dividing the proof into two cases.

Case (i) $\theta_k=\theta^0_k$. Since $\theta^0_k \le \bar\aalpha$, it follows that 
$\theta_k \le \bar\aalpha$ and the conclusion holds.

Case (ii) $\theta_k=\theta^0_k \eta^j$ for some integer $j>0$.  Suppose for contradiction 
that $\theta_k > c$. By \eqref{Lf-pmin} and \eqref{c}, we then have 
\beq \label{t-theta}
\tilde\theta_k := \theta_k/\eta > c/\eta \ge L_{\max}+\sigma \ge L_{i_k} + \sigma.
\eeq
Let $d \in \Re^N$ such that $d_\Bi=0$ for $i\neq i_k$ and 
\beq \label{dik}
d_{i_k}=\arg\min_s \left\{\nabla_{i_k} f(x^k)^T s + \frac{\tilde\theta_k}{2}
\|s\|^2 + \q_{i_k}(x^k_{i_k}+s)\right\}.
\eeq
It follows that 
\[
\nabla_{i_k} f(x^k)^T d_{i_k} +  \frac{\tilde\theta_k}{2}
\|d_{i_k}\|^2  + \q_{i_k}(x^k_{i_k}+d_{i_k})-\q_{i_k}(x^k_{i_k}) \le 0.
\]
Also, by \eqref{lk} and the definitions of $\theta_k$ and $\tilde\theta_k$, one knows that 
\beq \label{Fxd-ineq}
F(x^k + d)  >  F(x^{\ell(k)}) -\frac{\sigma}{2}  \|d\|^2.
\eeq
On the other hand, using  \eqref{lipineq}, \eqref{lk}, \eqref{t-theta}, 
\eqref{dik} and the definition of $d$, we have
\[
\ba{lcl}
F(x^k + d) &=& f(x^k+d) + \q(x^k+d) \ \le \ f(x^k) + \nabla_{i_k} f(x^k)^T d_{i_k} + \frac{L_{i_k}}{2}\|d_{i_k}\|^2 + \q(x^k+d) \\ [12pt]
&=& F(x^k) + \underbrace{\nabla_{i_k} f(x^k)^T d_{i_k} +  \frac{\tilde\theta_k}{2}
\|d_{i_k}\|^2  + \q_{i_k}(x^k_{i_k}+d_{i_k})-\q_{i_k}(x^k_{i_k})}_{\le 0} + \frac{L_{i_k}-\tilde\theta_k}{2}\|d_{i_k}\|^2 \\ [12pt]
& \le & F(x^k) + \frac{L_{i_k}-\tilde\theta_k}{2}\|d_{i_k}\|^2  \ \le \  F(x^{\ell(k)}) -\frac{\sigma}{2}  \|d\|^2,
\ea
\]
which is a contradiction to \eqref{Fxd-ineq}. Hence, $\theta_k \le c$ and 
the conclusion holds. 

(ii) Let $\theta^{k,i}$ be defined above. It follows from statement (i) that 
$\underline\aalpha \le \theta_{k,i} \le  c$, which together with the definition of $\Theta_k$ implies that statement (ii) holds.
\end{proof}

\gap

The next result provides some bound on the norm of a proximal gradient, which will be used in the 
subsequent analysis on convergence rate of RNBPG. 

\begin{lemma} 
Let $\{x^k\}$ be generated by RNBPG, $\bd^k$ and $c$ defined in \eqref{bdk} and \eqref{c}, 
respectively, and  
\beq \label{hg}
\hg^k = \arg\min_d \left\{\nabla f(x^k)^T d + \frac{1}{2} \|d\|^2 + \q(x^k+d)\right\}.
\eeq
Assume that $\Psi$ is convex. There holds
\beq \label{hdk-bdd}
\|\hg^k\| \le \frac{c}{2}\left[1+\frac{1}{\underline\aalpha} + \sqrt{1-\frac{2}{c} + \frac{1}{\underline\aalpha^2}}\ \right] 
\|\bd^k\|.
\eeq

\end{lemma}

\begin{proof}
The conclusion of this lemma follows from \eqref{bdk}, \eqref{hg}, Lemma \ref{approx-lip} (ii), and 
\cite[Lemma 3.5] {LuZh12} with $H=\Theta_k$, $\tilde H=I$, $Q=\Theta^{-1}_k$, $d=\bd^k$ 
and $\td = \hg^k$.
\end{proof}

We note that by the definition in~\eqref{c}, we have 
$c\geq\overline{\aalpha}>\underline{\aalpha}$, which implies
\[
  1-\frac{2}{c}+\frac{1}{\underline{\aalpha}^2}
  =\left(1-\frac{1}{\underline{\aalpha}}\right)^2 + \frac{2}{\underline{\aalpha}}-\frac{2}{c} > 0.
\]
Therefore, the expression under the square root in~\eqref{hdk-bdd}
is always positive.

\subsection{Convergence of expected objective value}

In this subsection we show that the sequence of expected objective values generated by the 
method converge to the expected limit of the objective values obtained by a random single run of the method.

The following lemma studies uniform continuity of the expectation of $F$ with respect to  
random sequences. 

\begin{lemma} \label{diff-randseqs}
Suppose that $F$ is uniform continuous in some $S \subseteq \dom(F)$.  Let $y^k$ and $z^k$ 
be two random vectors in $S$ generated from $\xi_{k-1}$. Assume that there exists $C>0$ such that $|F(y^k)-F(z^k)| \le C$ for all $k$, and moreover, 
\[
\lim_{k \to \infty} \bE_{\xi_{k-1}}[\|y^k-z^k\|] =0.
\]
Then there hold
\[
\lim_{k \to \infty} \bE_{\xi_{k-1}} [|F(y^k) - F(z^k)|] = 0,  \quad\quad\quad \lim_{k \to \infty} \bE_{\xi_{k-1}} [F(y^k) - F(z^k)] = 0.
\]
\end{lemma}

\begin{proof}
Since $F$ is uniformly continuous in $S$, it follows that given any $\epsilon>0$, there exists 
$\delta_\epsilon >0$ such that $|F(x)-F(y)| < \epsilon/2$ for all $x,y \in S$ satisfying 
$\|x-y\| < \delta_\epsilon$. Using these relations, the Markov inequality, and the assumption that 
$|F(y^k)-F(z^k)| \le C$ for all $k$ and  $\lim_{k \to \infty} \bE_{\xi_{k-1}}[\|\Delta^k\|] =0$, 
where $\Delta^k = y^k-z^k$, we obtain that for sufficiently large $k$, 
\[
\ba{lcl}
\bE_{\xi_{k-1}}[|F(y^k) - F(z^k)|] 
&=&   \bE_{\xi_{k-1}}\left[|F(y^k) - F(z^k)| \bigm| \|\Delta^k\| \ge \delta_\epsilon\right] \P(\|\Delta^k\|\ge \delta_\epsilon) \\ [8pt]
 & & +  \ \bE_{\xi_{k-1}}
\left[|F(y^k) - F(z^k)| \bigm|\|\Delta^k\| <\delta_\epsilon\right] \P(\|\Delta^k\| < \delta_\epsilon) \\ [8pt]
&\le& \frac{C\bE_{\xi_{k-1}}[\|\Delta^k\|]}{\delta_\epsilon} + \frac{\epsilon}{2}  \ \le \ \epsilon.
\ea
\]
 Due to the arbitrarily of $\epsilon$, we see that the first statement of 
this lemma holds. The second statement  immediately follows from the first 
statement and the well-known inequality 
\[
|\bE_{\xi_{k-1}} [F(y^k) - F(z^k)]| \ \le \ \bE_{\xi_{k-1}}\left[|F(y^k) - F(z^k)|\right].
\]
\end{proof}

\gap

We are ready to establish the first main result, that is, the expected objective values 
generated by the RNBPG method converge to the expected limit of the objective values 
obtained by a random single run of the method.

\begin{theorem} \label{function-expect}
Let $\{x^k\}$ and $\{d^k\}$ be the sequences generated by the RNBPG method. Assume that $F$ 
is uniform continuous in $\Omega(x^0)$, where $\Omega(x^0)$ is defined in \eqref{omega}.
Then the following statements hold:
\bi
\item[(i)]  $\lim_{k \to \infty} [\|d^k\|] = 0$ and  $\lim_{k \to \infty} F(x^k)  = F^*_{\xi_\infty}$ for some $F^*_{\xi_\infty} \in \Re$, where $\xi_\infty = \{i_1,i_2, \cdots\}$.
\item[(ii)] $\lim_{k \to \infty} \bE_{\xi_{k}} [\|d^k\|] = 0$ and 
\beq \label{limit-expfun}
\lim_{k \to \infty} \bE_{\xi_{k-1}} [F(x^k)] =  \lim_{k \to \infty} \bE_{\xi_{k-1}}[F(x^{\ell(k)})] = \bE_{\xi_\infty}[F^*_{\xi_\infty}].
\eeq 
\ei
\end{theorem}

\begin{proof}
By \eqref{reduct} and \eqref{lk}, we have
\beq \label{reduct-k}
F(x^{k+1}) \ \le \  F(x^{\ell(k)}) - \frac{\sigma}{2} \|d^k\|^2\quad\quad \forall k \ge 0.
\eeq
Hence, $F(x^{k+1}) \le  F(x^{\ell(k)})$, which together with \eqref{lk} implies that $F(x^{\ell(k+1)})  \le F(x^{\ell(k)})$. It then follows that
\[
\bE_{\xi_k}[F(x^{\ell(k+1)})] \ \le \ \bE_{\xi_{k-1}}[F(x^{\ell(k)})] \quad\quad\forall k \ge 1.
\]
Hence, $\{F(x^{\ell(k)})\}$ and $\{\bE_{\xi_{k-1}}[F(x^{\ell(k)})]\}$ are non-increasing. Since $F$ is bounded below, so are $\{F(x^{\ell(k)})\}$ 
and $\{\bE_{\xi_{k-1}}[F(x^{\ell(k)})]\}$. It follows that there exist some $F^*_{\xi_\infty}$, $\tF^*\in\Re$ such that
\beq \label{limit-0}
\lim\limits_{k \to \infty} F(x^{\ell(k)}) = F^*_{\xi_\infty}, \quad\quad
\lim\limits_{k \to \infty} \bE_{\xi_{k-1}}[F(x^{\ell(k)})] =\tF^*.
\eeq

We first show by induction that the following relations hold for all $j \ge 1$:
\beqa
 &\lim\limits_{k\to \infty} \|d^{\ell(k)-j}\| = 0,  \quad\quad 
&\lim\limits_{k\to \infty} F(x^{\ell(k)-j}) = F^*_{\xi_\infty}. \label{dF-limit-s} \\
&\lim\limits_{k\to \infty} \bE_{\xi_{k-1}}[\|d^{\ell(k)-j}\|] = 0, \quad\quad 
& \lim\limits_{k\to \infty} \bE_{\xi_{k-1}}[F(x^{\ell(k)-j})] = \tF^*. \label{dF-limit}
\eeqa
Indeed, replacing $k$ by $\ell(k)-1$ in \eqref{reduct-k}, we obtain that
\[
F(x^{\ell(k)})  \ \le \  F(x^{\ell(\ell(k)-1)}) - \frac{\sigma}{2}  \|d^{\ell(k)-1}\|^2 \quad\quad \forall k \ge M+1,
\]
which together with $\ell(k) \ge k-M$ and monotonicity of $\{F(x^{\ell(k)})\}$ yields
\beq \label{F-iter-s}
F(x^{\ell(k)})  \ \le \  F(x^{\ell(k-M-1)}) - \frac{\sigma}{2} \|d^{\ell(k)-1}\|^2\quad\quad \forall k \ge M+1.
\eeq
Then we have 
\beq \label{F-iter}
\bE_{\xi_{k-1}}[F(x^{\ell(k)})]  \ \le \  \bE_{\xi_{k-1}}[F(x^{\ell(k-M-1)})] - \frac{\sigma}{2} \bE_{\xi_{k-1}}[\|d^{\ell(k)-1}\|^2] \quad\quad \forall k \ge M+1.
\eeq
Notice that
\[
\bE_{\xi_{k-1}}[F(x^{\ell(k-M-1)})] \ =\ \bE_{\xi_{k-M-2}}[F(x^{\ell(k-M-1)})] \quad\quad\forall k \ge M+1.
\]
It follows from this relation and \eqref{F-iter} that
\beq \label{F-iter-1}
\bE_{\xi_{k-1}}[F(x^{\ell(k)})] \ \le \ \bE_{\xi_{k-M-2}}[F(x^{\ell(k-M-1)})] -
\frac{\sigma}{2} \bE_{\xi_{k-1}}[\|d^{\ell(k)-1}\|^2]\quad\quad \forall k \ge M+1.
\eeq
In view of \eqref{limit-0}, \eqref{F-iter-s}, \eqref{F-iter-1}, and  
$ (\bE_{\xi_{k-1}}[\|d^{\ell(k)-1}\|])^2 \le \bE_{\xi_{k-1}}[\|d^{\ell(k)-1}\|^2]$,  one can have
\beq \label{limit-d}
\lim\limits_{k\to\infty} \|d^{\ell(k)-1}\| = 0, \quad\quad \lim\limits_{k\to\infty} \bE_{\xi_{k-1}}[\|d^{\ell(k)-1}\|] = 0.
\eeq
One can also observe that $F(x^k) \le F(x^0)$ and hence $\{x^k\} \subset \Omega(x^0)$. 
Using this fact,  \eqref{limit-0}, \eqref{limit-d}, Lemma \ref{diff-randseqs}, and uniform continuity 
of $F$ over $\Omega(x^0)$, we obtain that
\[
\ba{l}
\lim\limits_{k\to\infty} F(x^{\ell(k)-1}) \ = \ \lim\limits_{k\to\infty}F(x^{\ell(k)}) \ = \ F^*_{\xi_\infty}, \\ [14pt]
\lim\limits_{k\to\infty}\bE_{\xi_{k-1}} [F(x^{\ell(k)-1})] \ = \ \lim\limits_{k\to\infty}\bE_{\xi_{k-1}} [F(x^{\ell(k)})] \ = \ \tF^*.
\ea
\]
Therefore, \eqref{dF-limit-s} and \eqref{dF-limit} hold for $j=1$. Suppose now that they hold for 
some $j \ge 1$. We need to show that they also hold for $j+1$. Replacing $k$ by $\ell(k)-j-1$ in 
\eqref{reduct-k} gives
\[
F(x^{\ell(k)-j})  \ \le \ F(x^{\ell(\ell(k)-j-1)}) - \frac{\sigma}{2} \|d^{\ell(k)-j-1}\|^2\quad\quad \forall k \ge M+j+1.
\]
By this relation, $\ell(k) \ge k-M$, and monotonicity of $\{F(x^{\ell(k)})\}$, one can have
\beq \label{F-iterj-s}
F(x^{\ell(k)-j})  \ \le \  F(x^{\ell(k-M-j-1)}) - \frac{\sigma}{2} \|d^{\ell(k)-j-1}\|^2\quad\quad \forall k \ge M+j+1.
\eeq
Then we obtain that 
\[
\bE_{\xi_{k-1}}[F(x^{\ell(k)-j})] \ \le \  \bE_{\xi_{k-1}}[F(x^{\ell(k-M-j-1)})] - \frac{\sigma}{2} \|d^{\ell(k)-j-1}\|^2\quad\quad \forall k \ge M+j+1.
\]
Notice that
\[
\bE_{\xi_{k-1}}[F(x^{\ell(k-M-j-1)})] \ = \ \bE_{\xi_{k-M-j-2}}[F(x^{\ell(k-M-j-1)})]\quad\quad \forall k \ge M+j+1.
\]
It follows from these two relations that
\beq \label{F-iterj}
\bE_{\xi_{k-1}}[F(x^{\ell(k)-j})] \le \bE_{\xi_{k-M-j-2}}[F(x^{\ell(k-M-j-1)})] -
\frac{\sigma}{2} \bE_{\xi_{k-1}}[\|d^{\ell(k)-j-1}\|^2],\quad\forall k \ge M+j+1.
\eeq
Using \eqref{limit-0}, \eqref{F-iterj-s}, \eqref{F-iterj}, the induction hypothesis, and a similar 
argument as above, we can obtain that
\[
\lim\limits_{k\to\infty}\|d^{\ell(k)-j-1}\| = 0, \quad\quad \lim\limits_{k\to\infty} \bE_{\xi_{k-1}}[\|d^{\ell(k)-j-1}\|] = 0.
\]
These relations, together with Lemma \ref{diff-randseqs}, uniform continuity 
of $F$ over $\Omega(x^0)$ and the induction hypothesis, yield
\[
\ba{l}
\lim\limits_{k\to\infty} F(x^{\ell(k)-j-1}) \ = \ \lim\limits_{k\to\infty} F(x^{\ell(k)-j}) \ = \ F^*_{\xi_\infty}, \\ [14pt]
\lim\limits_{k\to\infty}\bE_{\xi_{k-1}} [F(x^{\ell(k)-j-1})] \ = \ \lim\limits_{k\to\infty}\bE_{\xi_{k-1}} [F(x^{\ell(k)-j})] \ = \ \tF^*.
\ea
\]
Hence, \eqref{dF-limit-s} and \eqref{dF-limit} hold for $j+1$, and the proof of \eqref{dF-limit-s} and \eqref{dF-limit} is completed.

For all $k \ge 2M+1$, we define
\[
\td^{\ell(k)-j} = \left\{
\ba{ll}
d^{\ell(k)-j}   & \ \mbox{if} \ j \le \ell(k) - (k-M-1), \\
0 & \  \mbox{otherwise},
\ea
\right. \quad\quad j=1, \ldots, M+1.
\]
It is not hard to observe that
\beqa
\|\td^{\ell(k)-j}\| &\le & \|d^{\ell(k)-j}\|,  \label{td} \\
x^{\ell(k)} &=& x^{k-M-1} + \sum\limits_{j=1}^{M+1} \td^{\ell(k)-j}. \label{xlk}
\eeqa
It follows from \eqref{dF-limit-s}, \eqref{dF-limit} and \eqref{td} that 
$\lim\limits_{k\to \infty} \|\td^{\ell(k)-j}\| = 0$ and 
$\lim\limits_{k\to \infty} \bE_{\xi_{k-1}}[\|\td^{\ell(k)-j}\|] = 0$ for  $j=1,\ldots, M+1$. 
Hence, 
\[
\lim\limits_{k\to \infty}\left\|\sum\limits_{j=1}^{M+1} \td^{\ell(k)-j}\right\| = 0, \quad\quad
\lim\limits_{k\to \infty} \bE_{\xi_{k-1}}\left[\left\|\sum\limits_{j=1}^{M+1} \td^{\ell(k)-j}\right\|\right] = 0.
\]
These, together with \eqref{dF-limit-s}, \eqref{dF-limit}, \eqref{xlk}, Lemma \ref{diff-randseqs} 
and uniform continuity of $F$ over $\Omega(x^0)$, imply that
\beqa
\lim\limits_{k\to\infty} F(x^{k-M-1}) &=&  \lim\limits_{k\to\infty}
F(x^{\ell(k)})  =  F^*_{\xi_\infty}, \label{Fx-lim-s} \\ [4pt]
\lim\limits_{k\to\infty} \bE_{\xi_{k-1}}[F(x^{k-M-1})] & = & \lim\limits_{k\to\infty}
\bE_{\xi_{k-1}} [F(x^{\ell(k)})]  =  \tF^*. \label{Fx-lim}
\eeqa
It follows from \eqref{Fx-lim-s} that $\lim\limits_{k\to\infty} F(x^k) =  F^*_{\xi_\infty}$. 
Using this, \eqref{reduct-k} and \eqref{limit-0}, one can see that $\lim_{k\to \infty}\|d^k\|=0$. 
Hence, statement (i) holds. Notice that $\bE_{\xi_{k-M-2}}[F(x^{k-M-1})]=\bE_{\xi_{k-1}}[F(x^{k-M-1})]$. 
Combining this relation with \eqref{Fx-lim},  we have
\[
\lim\limits_{k\to\infty}\bE_{\xi_{k-M-2}}[F(x^{k-M-1})] \ = \ \tF^*,
\]
which is equivalent to
\[
\lim\limits_{k\to\infty} \bE_{\xi_{k-1}}[F(x^k)] \ = \ \tF^*.
\]
In addition, it follows from \eqref{reduct-k} that
\beq \label{descent-k}
\bE_{\xi_k} [F(x^{k+1})] \ \le \  \bE_{\xi_k}[F(x^{\ell(k)})] - \frac{\sigma}{2} \bE_{\xi_k}[\|d^k\|^2] 
\quad\quad \forall k \ge 0.
\eeq
Notice that 
\beq \label{lim-expfuns}
\lim\limits_{k\to\infty} \bE_{\xi_k}[F(x^{\ell(k)})] \ =\ 
\lim\limits_{k\to\infty}\bE_{\xi_{k-1}}[F(x^{\ell(k)})] \ \ = \ \tF^* \ = \ \lim\limits_{k\to\infty} \bE_{\xi_{k}}[F(x^{k+1})].
\eeq
Using \eqref{descent-k} and \eqref{lim-expfuns},  we conclude that $\lim_{k\to \infty}\bE_{\xi_k} [\|d^k\|]=0$. 

Finally, we claim that $\lim_{k \to \infty}\bE_{\xi_{k-1}} [F(x^k)]  = \bE_{\xi_\infty}[F^*_{\xi_\infty}]$. 
Indeed,  we know that $\{x^k\} \subset \Omega(x^0)$. Hence, $F^* \le F(x^k) \le F(x^0)$, where 
$F^*=\min_x F(x)$.  It follows that 
\[
|F(x^k)| \le \max\{|F(x^0)|, |F^*|\} \quad\quad \forall k.
\]
Using this relation and dominated convergence theorem (see, for example, \cite[Theorem 5.4]{Bil95}), we have
\[
\lim_{k \to \infty} \bE_{\xi_{\infty}} [F(x^k)] = \bE_{\xi_{\infty}}\left[\lim_{k \to \infty} F(x^k)\right] = 
\bE_{\xi_\infty}\left[F^*_{\xi_\infty}\right],
\]
which, together with $\lim_{k \to \infty}\bE_{\xi_{k-1}} [F(x^k)] = \lim_{k \to \infty}\bE_{\xi_{\infty}} [F(x^k)]$, implies that $\lim_{k \to \infty}\bE_{\xi_{k-1}} [F(x^k)]  = \bE_{\xi_\infty}[F^*_{\xi_\infty}]$. Hence, statement (ii) holds.

\end{proof}

\subsection{Convergence to stationary points}

In this subsection we show that when $k$ is sufficiently large, $x^k$ is an approximate 
stationary point of \eqref{NLP}  with high probability. 

\begin{theorem} \label{stationary-pt}
Let $\{x^k\}$ be generated by RNBPG, and $\bd^k$ and $\bx^k$ defined in \eqref{bdx}.  
Assume that $F$ is uniformly continuous and $\Psi$ is locally Lipschitz continuous in 
$\Omega(x^0)$, where $\Omega(x^0)$ is defined in \eqref{omega}.  Then there hold
\bi
\item[(i)]
\beq \label{statement-1}
\lim\limits_{k\to\infty}\bE_{\xi_{k-1}}[\|\bd^k\|] = 0, \quad\quad\quad
\lim\limits_{k\to\infty} \bE_{\xi_{k-1}} [\dist(-\nabla f(\tx^k), \partial \q(\tx^k)] \ = \ 0,
\eeq
where $\partial \Psi$ denotes the Clarke subdifferential of $\Psi$. 
\item[(ii)] Any accumulation point of $\{x^k\}$ is a stationary point of problem \eqref{NLP} almost surely.
\item[(iii)] Suppose further that $F$ is uniformly continuous in 
\beq \label{cS}
\cS =\left\{x: F(x) \le F(x^0)+\max\left\{\frac{n}{\sigma}\left|L_f -\underline\aalpha\right|, 1\right\} (F(x^0)-F^*)\right\}.
\eeq
Then $\lim_{k\to\infty}\bE_{\xi_{k-1}}[|F(x^k)-F(\tx^k)|]=0$. Moreover, for any $\epsilon>0$ and $\rho \in (0,1)$, there exists $K$ such that  for all $k \ge K$, 
\[
\P\left(\max\left\{\|x^k - \tx^k\|, |F(x^k)-F(\tx^k)|, \dist(-\nabla f(\tx^k), \partial \q(\tx^k))\right\}  \le  \eps \right) \ \ge  \ 1-\rho.  
\]
\ei
\end{theorem}

\begin{proof}
(i) We know from Theorem \ref{function-expect} (ii) that $\lim_{k\to\infty}\bE_{\xi_k}[\|d^k\|]=0$,  
which together with \eqref{bd-norm} implies $\lim_{k\to\infty}\bE_{\xi_{k-1}}[\|\bd^k\|] = 0$. 
Notice that $\bd^k$ is an optimal solution of problem \eqref{bdk}. By the first-order optimality condition (see, for example, Proposition 2.3.2 of \cite{Clarke90}) of \eqref{bdk} and $\tx^k = x^k+\bd^k$, one can have
\beq \label{opt-cond}
0 \in \nabla f(x^k) + \Theta_k\bd^k + \partial \q(\tx^k).
\eeq
In addition, it follows from \eqref{g-lip} that 
\[
\|\nabla f(\tx^k)-\nabla f(x^k)\| \ \le \  L_f \|\bd^k\|.
\]
Using this relation along with Lemma \ref{approx-lip} (ii) and \eqref{opt-cond}, we obtain that
\[
\dist(-\nabla f(\tx^k), \partial \q(\tx^k)) \ \le \ \left(c +L_f\right) \|\bd^k\|,
\]
which together with the first relation of \eqref{statement-1} implies that the second relation of 
\eqref{statement-1} also holds.

(ii) Let $x^*$ be an accumulation point of $\{x^k\}$. There exists a subsequence $\cK$ such that 
$\lim_{k \in \cK \to \infty} x^k = x^*$. Since $\bE_{\xi_{k-1}}[\|\bd^k\|] \to 0$, it follows that 
$\{\bd^k\}_{k \in \cK} \to 0$ almost surely. This together with the second relation of 
\eqref{statement-1} and outer semi-continuity of $\partial \Psi$ yields
\[
\dist(-\nabla f(x^*), \partial \q(x^*)) = \lim_{k \in \cK \to \infty} \dist(-\nabla f(\tx^k), \partial \q(\tx^k)) =0
\] 
almost surely. Hence, $x^*$ is a stationary point of problem \eqref{NLP} almost surely.

(iii) Recall that $\tx^k = x^k+\bd^k$.  It follows from \eqref{g-lip} that 
\[
f(\tx^k) \le  f(x^k) + \nabla f(x^k)^T \bd^k + \frac{1}{2}  L_f \|\bd^k\|^2.
\]
Using this relation and Lemma \ref{approx-lip} (ii), we have
\beqa
F(\tx^k) &\le&  f(x^k) + \nabla f(x^k)^T \bd^k + \frac{1}{2}  L_f \|\bd^k\|^2 + \q(x^k+\bd^k) \nn \\ 
&\le&  f(x^k)  + \nabla f(x^k)^T \bd^k + \frac{1}{2} (\bd^k)^T \Theta_k \bd^k + \q(x^k+\bd^k) +  \frac{1}{2}  \left(L_f -\underline\aalpha\right)\|\bd^k\|^2. \label{Fxbar}
\eeqa
In view of \eqref{bdk}, one has 
\[
\nabla f(x^k)^T \bd^k + \frac{1}{2} (\bd^k)^T \Theta_k \bd^k + \q(x^k+\bd^k) \le \q(x^k),
\]
which together with \eqref{Fxbar} yields
\[
F(\tx^k) \le  F(x^k) + \frac{1}{2}  (L_f -\underline\aalpha)\|\bd^k\|^2.
\]
Using this relation and the fact that $F(\tx^k) \ge F^*$ and $F(x^k) \le F(x^0)$, one 
can obtain that
\beq \label{diff-F}
|F(\tx^k)-F(x^k)| \le \max\left\{ \frac{1}{2}  \left|L_f -\underline\aalpha\right|\|\bd^k\|^2, F(x^0)-F^*\right\} \quad\quad \forall k.
\eeq
In addition, since $F^{l(k)} \le F(x^0)$ and $F(\tx^k) \ge F^*$, it follows from \eqref{F-reduct} that 
$\|\bd^{k,i}\|^2 \le 2(F(x^0)-F^*)/\sigma$. Hence, one has
\[
\|\bd^k\|^2=\sum^n_{i=1} \|\bd^{k,i}\|^2 \le 2n(F(x^0)-F^*)/\sigma\quad\quad \forall k.
\]
This inequality together with \eqref{diff-F} yields
\[
|F(\tx^k)-F(x^k)| \le \max\left\{\frac{n}{\sigma}\left|L_f -\underline\aalpha\right|, 1\right\} (F(x^0)-F^*) \quad\quad \forall k,
\]
and hence $\{|F(\tx^k)-F(x^k)|\}$ is bounded. Also, this inequality together with $F(x^k) \le F(x^0)$ 
and the definition of $\cS$ implies that $\tx^k$, $x^k \in \cS$ for all $k$. In addition, by 
statement (i), we know $\bE_{\xi_{k-1}}[\|x^k-\tx^k\|] \to 0$. In view of these facts and 
invoking Lemma \ref{diff-randseqs}, one has 
\beq \label{lim-diff-F}
\lim_{k \to \infty}\bE_{\xi_{k-1}}[|F(x^k)-F(\tx^k)|] \ = \ 0.
\eeq
Observe that 
\[
\ba{lcl}
0 &\le&  \max\left\{\|x^k - \tx^k\|, |F(x^k)-F(\tx^k)|, \dist(-\nabla f(\tx^k), \partial \q(\tx^k))\right\} \\ [6pt]
&\le & \|x^k - \tx^k\| + |F(x^k)-F(\tx^k)|+ \dist(-\nabla f(\tx^k), \partial \q(\tx^k)).
\ea
\]
Using these inequalities, \eqref{lim-diff-F} and statement (i), we see that
\[
\lim_{k \to\infty}\bE_{\xi_{k-1}} \left[\max\left\{\|x^k - \tx^k\|, |F(x^k)-F(\tx^k)|, \dist(-\nabla f(\tx^k), \partial \q(\tx^k))\right\} \right] \ = \ 0.
\]
The rest of statement (iii) follows from this relation and the Markov inequality.
\end{proof}

\subsection{Convergence rate analysis}


In this subsection we establish a sublinear rate of convergence of RNBPG 
in terms of the minimal expected squared norm of certain proximal gradients over the 
iterations.

\begin{theorem} \label{complexity}
Let $\bg^k = -\Theta_k \bd^k$,  $p_{\min}$, $\hg^k$ and $c$ be defined in \eqref{Lf-pmin}, 
\eqref{hg} and \eqref{c}, respectively, and $F^*$ the optimal value of \eqref{NLP}.  The 
following statements hold
\bi
\item[(i)] 
\[
\min\limits_{1 \le t \le k} \bE_{\xi_{t-1}}[\|\bg^t\|^2] \le  \frac{2c^2(F(x^0)-F^*)}
{\sigma p_{\min}}\cdot \frac{1}{\lfloor (k+1)/(M+1)\rfloor} \quad\quad \forall k \ge M.
\]
\item[(ii)] Assume further that $\Psi$ is convex. Then
\[
\min\limits_{1 \le t \le k} \bE_{\xi_{t-1}}[\|\hg^t\|^2] \le 
\frac{c^2(F(x^0)-F^*)}{2\sigma p_{\min}} 
 \left[1+\frac{1}{\underline\aalpha} + \sqrt{1-\frac{2}{c} + \frac{1}{\underline\aalpha^2}}\right]^2  \cdot \frac{1}{\lfloor (k+1)/(M+1)\rfloor} \quad\quad \forall k \ge M.
\]
\ei 
\end{theorem}

\begin{proof}
(i) Using $\bg^k = -\Theta_k \bd^k$, Lemma \ref{approx-lip} (ii), and \eqref{bd-sqr}, 
one can observe that 
\beq \label{dd-ineq} 
 \bE_{\xi_{k}}[\|d^k\|^2]  \ge   p_{\min} \bE_{\xi_{k-1}}[\|\bd^k\|^2] 
= p_{\min} \bE_{\xi_{k-1}}[\|\Theta^{-1}_k\bg^k\|^2]  \ge  \frac{p_{\min}} {c^2} \bE_{\xi_{k-1}}[\|\bg^k\|^2]. 
\eeq
Let $j(t) = l((M+1)t)-1$ and $\bj(t) = (M+1)t-1$ for all $t \ge 0$. One can see from \eqref{F-iter-1} that 
\[
\bE_{\xi_{\bj(t)}}[F(x^{j(t)+1})]  \ \le \  \bE_{\xi_{\bj(t-1)}}[F(x^{j(t-1)+1})] - \frac{\sigma}{2} \bE_{\xi_{\bj(t)}}[\|d^{j(t)}\|^2]\quad\quad \forall t \ge 1.
\]
Summing up the above inequality over $t=1, \ldots, s$, we have
\[
\bE_{\xi_{\bj(s)}}[F(x^{j(s)+1})]  \le  F(x^0) - \frac{\sigma}{2} \sum^s_{t=1} \bE_{\xi_{\bj(t)}}[\|d^{j(t)}\|^2] \le  F(x^0) - \frac{\sigma s}{2}\min\limits_{1 \le t \le s} \bE_{\xi_{\bj(t)}}[\|d^{j(t)}\|^2],
\]
which together with $\bE_{\xi_{\bj(s)}}[F(x^{j(s)+1})] \ge F^*$ implies that 
\beq \label{dnorm-bdd}
\min\limits_{1 \le t \le s} \bE_{\xi_{\bj(t)}}[\|d^{j(t)}\|^2] \le \frac{2(F(x^0)-F^*)}{\sigma s}. 
\eeq
Given any $k \ge M$, let $s_k=\lfloor (k+1)/(M+1)\rfloor$. Observe that 
\[
\bj(s_k) = (M+1)s_k-1 \le k.
\]
Using this relation and \eqref{dnorm-bdd}, we have
\[
\min\limits_{1 \le t \le k} \bE_{\xi_{t}}[\|d^t\|^2] \le \min\limits_{1 \le \tilde t \le s_k} \bE_{\xi_{\bj(\tilde t)}}[\|d^{j(\tilde t)}\|^2] \le \frac{2(F(x^0)-F^*)}{\sigma \lfloor (k+1)/(M+1)\rfloor} \quad\quad \forall k \ge M,
\] 
which together with \eqref{dd-ineq} implies that statement (i) holds.

(ii) It follows from  \eqref{bd-sqr} and \eqref{dnorm-bdd} that 
\[
\min\limits_{1 \le t \le s} \bE_{\xi_{\bj(t)-1}}[\|\bd^{j(t)}\|^2] \le \frac{2(F(x^0)-F^*)}
{\sigma s p_{\min}}. 
\]
Using this relation and a similar argument as above, one has 
\[
\min\limits_{1 \le t \le k} \bE_{\xi_{t-1}}[\|\bd^t\|^2] \le \min\limits_{1 \le \tilde t \le s_k} \bE_{\xi_{\bj(\tilde t)-1}}[\|\bd^{j(\tilde t)}\|^2] \le \frac{2(F(x^0)-F^*)}{\sigma p_{\min}\lfloor (k+1)/(M+1)\rfloor} \quad\quad \forall k \ge M.
\]
Statement (ii) immediately follows from this inequality and \eqref{hdk-bdd}.
\end{proof}

\section{Convergence analysis for structured convex problems}
\label{monotone}

In this section we study convergence of RNBPG for solving structured convex problem \eqref{NLP}. To this end, we assume throughout this section that $f$ and $\q$ are both convex functions.

The following result shows that $F(x^k)$ can be arbitrarily close to  the 
optimal value $F^*$ of \eqref{NLP} with high probability for sufficiently large $k$.

\begin{theorem} \label{convex-opt}
Let $\{x^k\}$ be generated by the RNBPG method, and let $F^*$ and $X^*$ the optimal value and 
the set of optimal solutions of \eqref{NLP}, respectively. Suppose that $f$ and $\q$ are convex 
functions and $F$ is uniformly continuous in $\cS$, where $\cS$ is defined in
\eqref{cS}. Assume that there exists a subsequence $\cK$ such that $\{\bE_{\xi_{k-1}}[\dist(x^k,X^*)]\}_{\cK}$ is bounded. Then there hold:
\bi
\item[(i)]
\[
\lim\limits_{k\to\infty} \bE_{\xi_{k-1}} [F(x^k)] = F^*.
\]
\item[(ii)] For any $\epsilon>0$ and $\rho \in (0,1)$, there exists $K$ such that  for all $k \ge K$, 
\[
\P\left(F(x^k)-F^* \ \le \ \eps \right) \ \ge  \ 1-\rho.  
\]
\ei
\end{theorem}

\begin{proof}
(i) Let $\bd^k$ be defined in\eqref{bdx}. Using the assumption that $F$ is uniformly continuous in 
$\cS$ and Theorem \ref{stationary-pt}, one has
\beqa 
& & \lim\limits_{k\to\infty}\bE_{\xi_{k-1}}[\|\bd^k\|] = 0,  \quad\quad
\lim\limits_{k\to\infty} \bE_{\xi_{k-1}} [\|s^k\|] \ = \ 0, \label{subgrad} \\
& & \lim\limits_{k\to\infty}\bE_{\xi_{k-1}}[F(x^k+\bd^k)] = \lim\limits_{k\to\infty}\bE_{\xi_{k-1}}[F(x^k)] = \tF^* \label{limit-Fbd}
\eeqa
for some $s^k \in \partial F(x^k+\bd^k)$ and $\tF^* \in \Re$.  Let $x^k_*$ be the projection of 
$x^k$ onto $X^*$. By the convexity of $F$, we have
\beq \label{upp-bdd}
F(x^k+\bd^k) \ \le \ F(x^k_*) + (s^k)^T(x^k+\bd^k-x^k_*).
\eeq
One can observe that
\[
\ba{lcl}
|\bE_{\xi_{k-1}}[(s^k)^T(x^k+\bd^k-x^k_*)]|  &\le & \bE_{\xi_{k-1}}[|(s^k)^T(x^k+\bd^k-x^k_*)|] \\ [10pt]
&\le & \bE_{\xi_{k-1}} [\|s^k\| \|(x^k+\bd^k-x^k_*)\|] \\ [10pt]
&\le & \sqrt{ \bE_{\xi_{k-1}} [\|s^k\|^2]}  \sqrt{\bE_{\xi_{k-1}} [\|(x^k+\bd^k-x^k_*)\|^2]} \\ [10pt]
&\le & \sqrt{ \bE_{\xi_{k-1}} [\|s^k\|^2]}  \sqrt{2\bE_{\xi_{k-1}} [\left(\dist(x^k,X^*)\right)^2+
\|\bd^k\|^2]},
\ea
\]
which, together with \eqref{subgrad} and the assumption that $\{\bE_{\xi_{k-1}}[\dist(x^k,X^*)]\}_{\cK}$ is bounded, implies that
\[
\lim\limits_{k\in \cK \to\infty}\bE_{\xi_{k-1}}[(s^k)^T(x^k+\bd^k-x^k_*)] = 0.
\]
Using this relation, \eqref{limit-Fbd} and \eqref{upp-bdd}, we obtain that
\[
\ba{lcl}
\tF^* &=& \lim\limits_{k\to\infty}\bE_{\xi_{k-1}}[F(x^k)] = \lim\limits_{k\to\infty}\bE_{\xi_{k-1}}[F(x^k+\bd^k)] \\ [10pt]
& = &\lim\limits_{k \in \cK \to\infty}\bE_{\xi_{k-1}}[F(x^k+\bd^k)] 
\le \lim\limits_{k \in \cK \to\infty}\bE_{\xi_{k-1}}[F(x^k_*)]  \ = \ F^*,
\ea
\]
which together with $\tF^* \ge F^*$ yields $\tF^* = F^*$. Statement (i) follows from 
this relation and \eqref{limit-Fbd}.
 
(ii) Statement (ii) immediately follows from statement (i), the Markov inequality, and the fact 
$F(x^k)\ \ge \ F^*$.
\end{proof}

\gap

 In the rest of this section we study the rate of convergence of a monotone version of RNBPG, i.e., 
$M=0$, or equivalently, \eqref{reduct} is replaced by 
\beq \label{reduct1}
F(x^k + d^k) \ \le \  F(x^k) - \frac{\sigma}{2}  \|d^k\|^2.
\eeq

The following lemma will be subsequently used to establish a sublinear rate of convergence of RNBPG 
with $M=0$. 

\begin{lemma} \label{recur-seq}
Suppose that a nonnegative sequence $\{\Delta_k\}$  satisfies 
\beq \label{Delta-t}
\Delta_k \ \le \ \Delta_{k-1} - \alpha
\Delta^2_k \quad\quad \forall k \ge 1
\eeq
for some $\alpha>0$. Then 
\[
\Delta_k \le \frac{\max\{2/\alpha, \Delta_0\}}{k+1} \quad\quad \forall k \ge 0.
\]
\end{lemma}

\begin{proof}
We divide the proof into two cases.  

Case (i): Suppose $\Delta_k >0$ for all $k \ge 0$. Let $\bDelta_k=1/\Delta_k$. It follows from \eqref{Delta-t} that 
\[
\bDelta^2_k -\bDelta_{k-1} \bDelta_k - \alpha \bDelta_{k-1} \ge 0 \quad\quad \forall k \ge 1,
\]
which together with $\bDelta_k > 0$ implies that 
\beq \label{bDelta-ineq1}
\bDelta_k \ge \frac{\bDelta_{k-1}+\sqrt{\bDelta^2_{k-1}+4\alpha \bDelta_{k-1}}}{2}.
\eeq
We next show by induction that 
\beq \label{bDelta-ineq2}
\bDelta_k \ge \beta (k+1) \quad\quad \forall k \ge 0,
\eeq
where $\beta = \min\left\{\alpha/2, \bDelta_0\right\}$. 
By the definition of $\beta$, one can see that \eqref{bDelta-ineq2} holds for $k=0$. Suppose it 
holds for some $k \ge 0$. We now need to show \eqref{bDelta-ineq2} also holds for $k+1$. Indeed, 
since $\beta \le \alpha/2$, we have 
\[
\alpha(k+1) \ge \alpha \left(k/2+1\right) = \alpha (k+2)/2 \ge \beta(k+2). 
\]
which yields
\[
4\alpha\beta(k+1) \ge \beta^2(4k+8) = [2\beta(k+2)-\beta(k+1)]^2 - \beta^2(k+1)^2.
\]
It follows that 
\[
\sqrt{\beta^2(k+1)^2+ 4\alpha\beta(k+1)} \ge 2\beta(k+2)-\beta(k+1),
\]
which is equivalent to
\[
\beta(k+1) + \sqrt{\beta^2(k+1)^2+ 4\alpha\beta(k+1)} \ge 2\beta(k+2).
\]
Using this inequality, \eqref{bDelta-ineq1} and the induction hypothesis $\bDelta_k \ge \beta (k+1)$, 
we obtain that
\[
\bDelta_{k+1} \ \ge \ \frac{\bDelta_k+\sqrt{\bDelta^2_k+4\alpha \bDelta_k}}{2} \ \ge \ 
 \frac{\beta(k+1) + \sqrt{\beta^2(k+1)^2+ 4\alpha\beta(k+1)}}{2} \ \ge \ \beta(k+2),
\]
namely, \eqref{bDelta-ineq2} holds for $k+1$. Hence, the induction is completed and \eqref{bDelta-ineq2} holds for all $k \ge 0$. The conclusion of this lemma follows from \eqref{bDelta-ineq2} and the definitions of $\bDelta_k$ and $\beta$.  

Case (ii) Suppose there exists some $\tilde k$ such that $\Delta_{\tilde k}=0$. Let $K$ be the smallest 
of such integers. Since $\Delta_k \ge 0$, it follows from \eqref{Delta-t} that $\Delta_k=0$ for all 
$k \ge K$ and $\Delta_k>0$ for every $0 \le k <K$. Clearly, the conclusion of this lemma holds for 
$k \ge K$. And it also holds for $0 \le k <K$ due to a similar argument as for Case (i).
\end{proof}

\gap

We next establish a sublinear rate of convergence on the expected objective values for the RNBPG 
method with $M=0$ when applied to problem \eqref{NLP}, where $f$ and $\psi$ are assumed 
to be convex. Before proceeding, we define the following quantities
\beqa
& & r = \max\limits_x \left\{\dist(x,X^*): x \in \Omega(x^0)\right\}, \label{r} \\ [5pt]
& & q =  \max\limits_x \left\{\|\nabla f(x)\|: x \in \Omega(x^0)\right\}, \label{q}
\eeqa
where $X^*$ denotes the set of optimal solutions of \eqref{NLP} and $\Omega(x^0)$ is defined in 
\eqref{omega}.

\begin{theorem} \label{sublinear}
Let $c, r, q$ be defined in \eqref{c}, \eqref{r}, \eqref{q}, respectively. Assume that $r$ and $q$ are finite. Suppose that $\Psi$ is $L_{\Psi}$-Lipschitz continuous in $\dom(\Psi)$, namely, 
\beq \label{L-Psi}
|\Psi(x)-\Psi(y)| \le L_{\Psi} \|x-y\| \quad\quad x, y \in \dom(\Psi)
\eeq
for some $L_{\Psi} >0$. Let $\{x^k\}$ be generated by RNBPG with $M=0$. Then
\[
\bE_{\xi_{k-1}}[F(x^k)]-F^*   \ \le \ \frac{\max\{2/\alpha, F(x^0)-F^*\}}{k+1} \quad\quad \forall k \ge 0,
\]
where 
\beq \label{alpha}
\alpha = \frac{\sigma p^2_{\min}}{2(L_{\Psi}+q+cr)^2}.  
\eeq
\end{theorem}

\begin{proof}
 Let $\bx^k$ be defined in \eqref{bdx}. For each $x^k$, let $x^k_*\in X^*$ such that $\|x^k-x^k_*\|=\dist(x^k,X^*)$. Due to $x^k \in  \Omega(x^0)$ and \eqref{r}, we know that $\|x^k-x^k_*\| \le r$. 
By the definition of $\bx^{k+1}$ and \eqref{bdk}, one can observe that 
\beq \label{bxk-opt}
[\nabla f(x^{k}) +\Theta_k(\bx^{k+1}-x^{k})]^T (\bx^{k+1}-x^k_*) + \Psi(\bx^{k+1}) - \Psi(x^k_*) \le 0.
\eeq
Using this inequality, \eqref{q}, and \eqref{L-Psi}, we have 
\[
\ba{l}
 F(x^{k}) - F^* = f(x^{k})-f(x^k_*) + \Psi(x^{k})-\Psi(\bx^{k+1})+\Psi(\bx^{k+1})-\Psi(x^k_*) \\ [8pt]
\le \nabla f(x^{k})^T(x^{k}-x^k_*) + L_{\Psi}\|x^{k}-\bx^{k+1}\| + \Psi(\bx^{k+1}) - \Psi(x^k_*) \\ [8pt]
=\nabla f(x^{k})^T(x^{k}-\bx^{k+1}) +  \nabla f(x^{k})^T(\bx^{k+1}-x^k_*) 
+L_{\Psi}\|x^{k}-\bx^{k+1}\| + \Psi(\bx^{k+1}) - \Psi(x^k_*) \\ [8pt]
\le( L_{\Psi}+q) \|x^{k}-\bx^{k+1}\| +(x^{k}-\bx^{k+1})^T \Theta_k(\bx^{k+1}-x^k_*) \\ [8pt]
\ \ \ + \underbrace{[\nabla f(x^{k}) +\Theta_k(\bx^{k+1}-x^{k})]^T (\bx^{k+1}-x^k_*)  
+ \Psi(\bx^{k+1}) - \Psi(x^k_*)}_{\le 0} \\ [22pt]
\le (L_{\Psi}+q) \|x^{k}-\bx^{k+1}\| + (x^{k}-\bx^{k+1})^T\Theta_k (\bx^{k+1}-x^k_*) \\ [8pt]
\le (L_{\Psi}+q) \|x^{k}-\bx^{k+1}\| + \underbrace{(x^{k}-\bx^{k+1})^T\Theta_k(\bx^{k+1}-x^{k})}_{\le 0} 
+ (x^{k}-\bx^{k+1})^T\Theta_k(x^{k}-x^k_*) \\ [22pt]
\le (L_{\Psi}+q) \|x^{k}-\bx^{k+1}\| + (x^{k}-\bx^{k+1})^T\Theta_k(x^{k}-x^k_*) \\ [8pt]
\le (L_{\Psi}+q) \|x^{k}-\bx^{k+1}\| + \|\Theta_k\| \|x^{k}-\bx^{k+1}\| \|x^{k}-x^k_*\| \\ [8pt] 
\le (L_{\Psi}+q+cr) \|x^{k}-\bx^{k+1}\| = (L_{\Psi}+q+cr) \|\bd^{k}\|,
\ea
\]
where the first inequality follows from convexity of $f$ and \eqref{L-Psi}, the second inequality is due to 
\eqref{q}, the third inequality follows from \eqref{bxk-opt}, and the last inequality is due to $\|x^k-x^k_*\| \le r$. The preceding inequality, \eqref{bd-norm} and  the fact $F(x^{k+1}) \le F(x^k)$ yield
\[
\bE_{\xi_k}[F(x^{k+1}]-F^* \le \bE_{\xi_{k-1}}[F(x^k)]-F^* \le (L_{\Psi}+q+cr) \bE_{\xi_{k-1}}[\|\bd^{k}\|] \le \frac{L_{\Psi}+q+cr}{p_{\min}} \bE_{\xi_{k-1}}[\|d^{k}\|].
\]
In addition, using $\left(\bE_{\xi_{k-1}}[\|d^{k}\|]\right)^2 \le \bE_{\xi_{k-1}}[\|d^{k}\|^2]$ and \eqref{reduct1}, one has
\[
\bE_{\xi_k}[F(x^{k+1})] \ \le \ \bE_{\xi_{k-1}}[F(x^k)] -  \frac{\sigma}{2} \bE_{\xi_{k-1}}[\|d^{k}\|^2] 
\ \le \ \bE_{\xi_{k-1}}[F(x^k)] -  \frac{\sigma}{2} \left(\bE_{\xi_{k-1}}[\|d^{k}\|]\right)^2.
\]
Let $\Delta_{k} = \bE_{\xi_{k-1}}[F(x^k)]-F^*$. Combining the preceding two inequalities, we obtain 
that
\[
\Delta_{k+1} \ \le \ \Delta_{k} - \alpha
\Delta^2_{k+1} \quad\quad \forall k \ge 0,
\]
where $\alpha$ is defined in \eqref{alpha}. Notice that $\Delta_0=F(x^0)-F^*$. Using this relation, 
the definition of $\Delta_k$, and Lemma \ref{recur-seq}, one can see that the conclusion of this 
theorem holds.
\end{proof}

\gap

The next result shows that under an error bound assumption the RNBPG method with $M=0$ is globally linearly convergent in terms of the expected objective values. 

%

\begin{theorem} 
Let $\{x^k\}$ be generated by RNBPG. Suppose that there exists $\tau >0$ such that 
\beq \label{err-bdd}
\dist(x^k,X^*) \le \tau \|\hg^k\| \quad\quad \forall k \ge 0,
\eeq
where $\hg^k$ is given in \eqref{hg} and $X^*$ denotes the set of optimal solutions 
of \eqref{NLP}. Then there holds
\[
\bE_{\xi_k}[F(x^k)] - F^* \le \left[\frac{2\varpi+(1-p_{\min}) \sigma}{2\varpi+\sigma}\right]^k (F(x^0)-F^*) \quad\quad \forall k \ge 0,
\]
where 
\[
\varpi =  \frac{(c+L_f) \tau^2 c^2}{8}\left[1+\frac{1}{\underline\aalpha} + \sqrt{1-\frac{2}{c} + \frac{1}{\underline\aalpha^2}}\ \right]^2 + \frac{L_{\max} - \underline \theta}{2}.
\]
\end{theorem}

\begin{proof}
For each $x^k$, let $x^k_*\in X^*$ such that $\|x^k-x^k_*\|=\dist(x^k,X^*)$. Let $\bd^k$ be 
defined in \eqref{bdx}, and 
\[
\Phi(\bd^k;x^k) = f(x^k)+\nabla f(x^k)^T \bd^k + \frac12\|\bd^k\|^2_{\Theta_k} + \Psi(x^k+\bd^k).
\]
 It follows from \eqref{g-lip} that 
\[
f(x+h) \ge f(x) + \nabla f(x)^T h - \frac12 L_f  \|h\|^2\quad\quad \forall x, h \in \Re^N.   
\] 
Using this inequality, \eqref{bdk} and Lemma \ref{approx-lip} (ii), we have that 
 \[
\ba{lcl}
\Phi(\bd^k;x^k) & \le & f(x^k)+\nabla f(x^k)^T (x^k_*-x^k) + \frac12\|x^k_*-x^k\|^2_{\Theta_k} + \Psi(x^k_*) \\ [6pt]
&\le & f(x^k_*) + \frac12 L_f \|x^k_*-x^k\|^2 + \frac12\|x^k_*-x^k\|^2_{\Theta_k} + \Psi(x^k_*) 
\\ [6pt]
&\le&  F(x^k_*) + \frac12 \gamma \|x^k_*-x^k\|^2 = F^* + \frac12 \gamma  [\dist(x^k,X^*)]^2.
\ea
\]
where $\gamma = c+L_f$. Using this relation and \eqref{err-bdd}, one can obtain that
\[
\Phi(\bd^k;x^k) \le F^* + \frac12 \gamma \tau^2 \|\hg^k\|^2.
\] 
It follows from this inequality and \eqref{hdk-bdd}  that
\[
\Phi(\bd^k;x^k) \le F^* + \frac18 \gamma \tau^2 c^2\left[1+\frac{1}{\underline\aalpha} + \sqrt{1-\frac{2}{c} + \frac{1}{\underline\aalpha^2}}\ \right]^2 
\|\bd^k\|^2,
\] 
which along with \eqref{bd-sqr} yields 
\beq \label{Ephi-ineq}
\bE_{\xi_{k-1}}[\Phi(\bd^k;x^k)] \le F^* + \frac{\gamma \tau^2 c^2}{8p_{\min}}\left[1+\frac{1}{\underline\aalpha} + \sqrt{1-\frac{2}{c} + \frac{1}{\underline\aalpha^2}}\ \right]^2 
\bE_{\xi_{k}}[\|d^k\|^2].
\eeq
In addition, by \eqref{lipineq} and the definition of $\bd^{k,i}$, we have
\beq \label{Fxik}
F(x^k+\bd^{k,i}) \le f(x^k) + \nabla f(x^k)^T \bd^{k,i} + \frac{L_i}{2}\|\bd^{k,i}\|^2 + \Psi(x^k+\bd^{k,i}) \quad\quad \forall i.
\eeq
It also follows from \eqref{subprob-soln} that 
\beq \label{1st-order}
\nabla f(x^k)^T \bd^{k,i} + \frac{\theta_{k,i}}{2} \|\bd^{k,i}\|^2 + \Psi(x^k+\bd^{k,i}) -\Psi(x^k) \le 0 \quad\quad \forall i.
\eeq
Using these two inequalities, we can obtain that
\[
\ba{lcl}
\bE_{i_k} [F(x^{k+1})] &=& \bE_{i_k} [F(x^k + \bd^{k,i_k})\bigm| \xi_{k-1}] \ = \ \sum^n_{i=1} p_i F(x^k + \bd^{k,i}) \\ [10pt] 
&\le & \sum^n_{i=1} p_i  [f(x^k) + \nabla f(x^k)^T \bd^{k,i} + \frac{L_i}{2}\|\bd^{k,i}\|^2 + \Psi(x^k+\bd^{k,i})] \\ [10pt] 
&=& F(x^k) + \sum^n_{i=1} p_i  [\nabla f(x^k)^T \bd^{k,i} + \frac{L_i}{2}\|\bd^{k,i}\|^2 + \Psi(x^k+\bd^{k,i})-\Psi(x^k)] \\ [10pt] 
&=& F(x^k) + \sum^n_{i=1} p_i  \underbrace{[\nabla f(x^k)^T \bd^{k,i} + \frac{\theta_{k,i}}{2}\|\bd^{k,i}\|^2 + \Psi(x^k+\bd^{k,i})-\Psi(x^k)]}_{\le 0} \\ [22pt] 
&& + \frac12\sum^n_{i=1}  p_i(L_i-\theta_{k,i})\|\bd^{k,i}\|^2 \\ [10pt] 
&\le& F(x^k) + p_{\min} \sum^n_{i=1} [\nabla f(x^k)^T \bd^{k,i} + \frac{\theta_{k,i}}{2}\|\bd^{k,i}\|^2 + \Psi(x^k+\bd^{k,i})-\Psi(x^k)] \\ [10pt] 
&& + \frac12\sum^n_{i=1}  p_i(L_i-\theta_{k,i})\|\bd^{k,i}\|^2 \\ [10pt] 
&=& F(x^k) + p_{\min} [\nabla f(x^k)^T \bd^k + \frac12\|\bd^k\|^2_{\Theta_k} + \Psi(x^k+\bd^k)-\Psi(x^k)] \\ [10pt] 
&& + \frac12\sum^n_{i=1}  p_i(L_i-\theta_{k,i})\|\bd^{k,i}\|^2 \\ [10pt] 
&\le& (1-p_{\min}) F(x^k) + p_{\min} \Phi(\bd^k;x^k) + \frac{L_{\max} - \underline \theta}{2} \ \bE_{i_k} [\|d^k\|^2\bigm| \xi_{k-1}], 
\ea
\]
where the first inequality follows from \eqref{Fxik} and the second inequality is due to 
\eqref{1st-order}. Taking expectation with respect to $\xi_{k-1}$ on both sides of the above 
inequality gives   
\[
\bE_{\xi_k} [F(x^{k+1})] \le (1-p_{\min}) \bE_{\xi_{k-1}}[F(x^{k})] + p_{\min} \bE_{\xi_{k-1}}[\Phi(\bd^k;x^k)] + \frac{L_{\max} - \underline \theta}{2} \ \bE_{\xi_k} [\|d^k\|^2].
\]
Using this inequality and \eqref{Ephi-ineq}, we obtain that 
\[
\bE_{\xi_k} [F(x^{k+1})] \le (1-p_{\min}) \bE_{\xi_{k-1}}[F(x^k)] + p_{\min} F^* + \varpi 
\bE_{\xi_{k}}[\|d^k\|^2] \quad\quad \forall k \ge 0,
\]
where $\varpi$ is defined above. In addition, it follows from \eqref{reduct1} that 
\[
\bE_{\xi_k}[F(x^{k+1})] \ \le \  \bE_{\xi_{k-1}}[F(x^k)] - \frac{\sigma}{2} \bE_{\xi_{k}}[\|d^k\|^2]\quad\quad \forall k \ge 0.
\]
Combining these two inequalities, we obtain that
\[
\bE_{\xi_k}[F(x^{k+1})] - F^* \le \frac{2\varpi+(1-p_{\min}) \sigma}{2\varpi+\sigma} \left(\bE_{\xi_{k-1}}[F(x^k)]-F^*\right) \quad\quad \forall k \ge 0,
\]
and the conclusion of this theorem immediately follows.
\end{proof}

\gap

\begin{remark}
The error bound condition \eqref{err-bdd} holds for a class of problems, especially when
$f$ is strongly convex. More discussion about this condition can be found, for example, in \cite{HoLu12}. 
\end{remark}

\section{Numerical experiments} \label{comp}

In this section we illustrate the numerical behavior of the RNBPG method on the 
$\ell_1$-regularized least-squares problem and a dual SVM problem in machine learning. 
First we consider the $\ell_1$-regularized least-squares problem:
\[
    F^*= \min_{x\in\Re^N} ~\left\{\frac{1}{2}\|Ax-b\|_2^2 + \lambda\|x\|_1 \right\},
\]
where $A\in \Re^{m\times N}$, $b\in\Re^m$, 
and $\lambda>0$ is a regularization parameter. 
Clearly, this problem is a special case of the general model~\eqref{NLP}  with 
$f(x) = \|Ax-b\|_2^2/2$ and $\Psi(x)=\lambda \|x\|_1$ and thus our 
proposed RNBPG method can be suitably applied to solve it. 

We generated a random instance with $m=1000$ and $N=2000$ following the 
procedure described in \cite[Section~6]{Nesterov13composite}. 
The advantage of this procedure is that an optimal solution $x^*$ 
is generated together with~$A$ and~$b$, and hence 
the optimal value $F^*$ is known. 
We generated an instance where the optimal solution $x^*$ has only
$200$ nonzero entries, so this can be considered as a sparse recovery problem.
We compare RNBPG with the following methods:
\begin{itemize}
    \item RBCD: The RBCD method \cite{RichtarikTakac12} with constant 
        step sizes $1/L_i$ determined by the Lipschitz constants $L_i$. 
        Here, $L_i=\|A_{:,i}\|_2^2$ where $A_{:,i}$ is the $i$th column
        block corresponding to the block partitions of~$x_i$ and $\|\cdot\|_2$
        is the matrix spectral norm.
    \item RBCD-LS: A variant of RBCD method with variable stepsizes that are 
        determined by a block-coordinate-wise backtracking line search scheme. 
        This method can also be regarded as a variant of RNBPG with $M=0$,
        but which has the property of monotone descent. 
\end{itemize}

As discussed in \cite{Nesterov12rcdm}, 
the structure of the least-squares function $f(x)=\|Ax-b\|_2^2/2$ allows
efficient computation of coordinate gradients, 
with cost of $O(mN_i)$ operations for block~$i$ as opposed to $O(mN)$ for 
computing the full gradient.
We note that the same structure also allows efficient computation of the 
function value, which costs the same order of operations as computing
coordinate gradients.
Therefore the backtracking line search used in RBCD-LS 
as well as the nonmonotone line search used in RNBPG 
(both relies on computing function values),
have the same order computational cost as evaluating coordinate gradients
at each iteration.
Therefore we can focus on comparing their required number of iterations to
obtain the same accuracy in reducing the objective value.

\begin{figure}[t]
  \psfrag{k}[cc]{\small Iteration number $k$}
  \psfrag{F}[bc]{\small $F(x^k)-F^*$}
  \psfrag{RBCD}[bl]{\footnotesize RBCD}
  \psfrag{RACD}[bl]{\footnotesize RACD}
  \psfrag{RBCD-LSLS}[bl]{\footnotesize RBCD-LS}
  \psfrag{RBCD-NM}[bl]{\footnotesize RNBPG}
  \begin{center}
    \subfloat[Blocksize $N_i=1$.]{
    \includegraphics[width=0.45\textwidth]{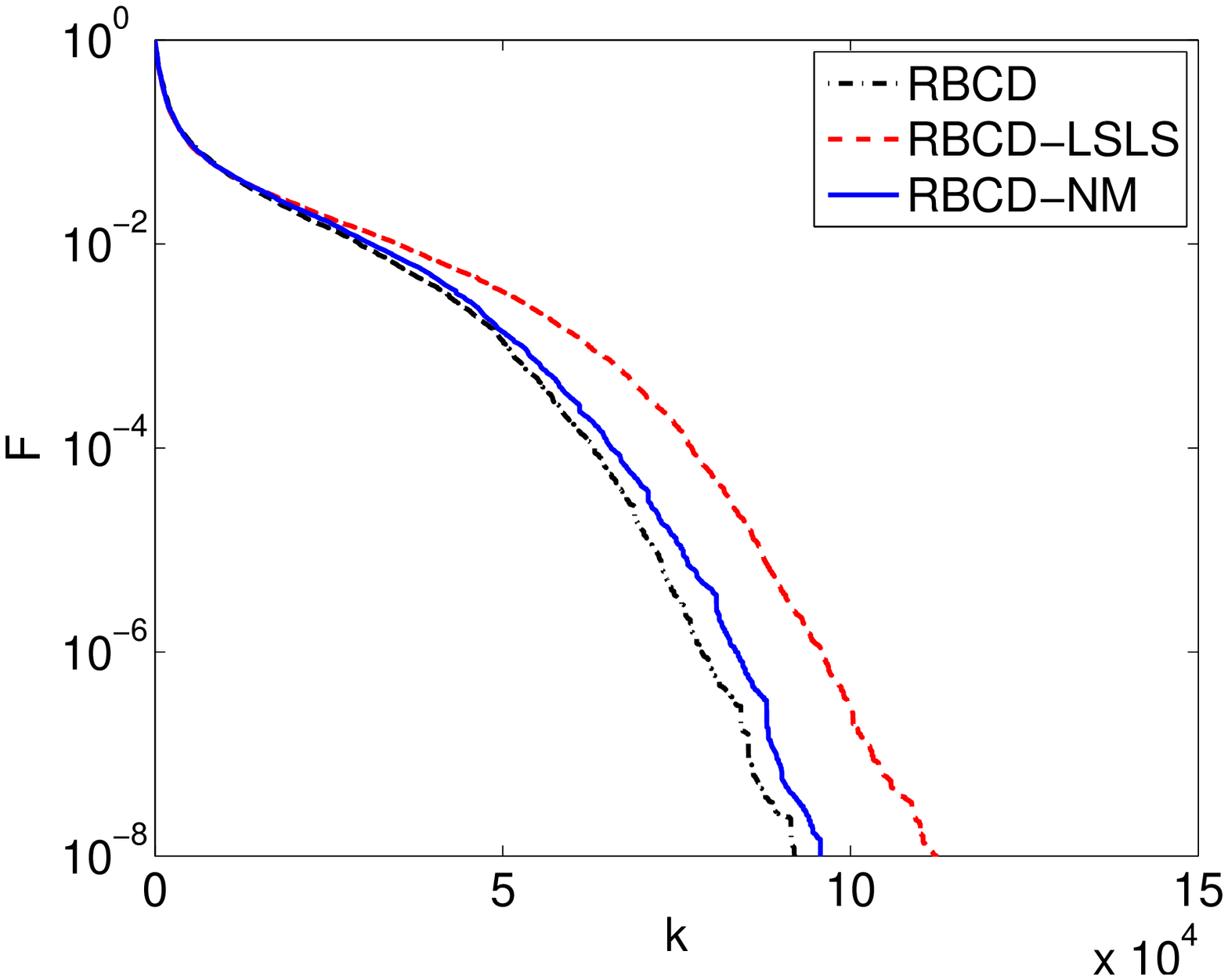}
    \label{fig:lasso-b1}}
    \hfill
    \subfloat[Blocksize $N_i=20$.]{
    \includegraphics[width=0.45\textwidth]{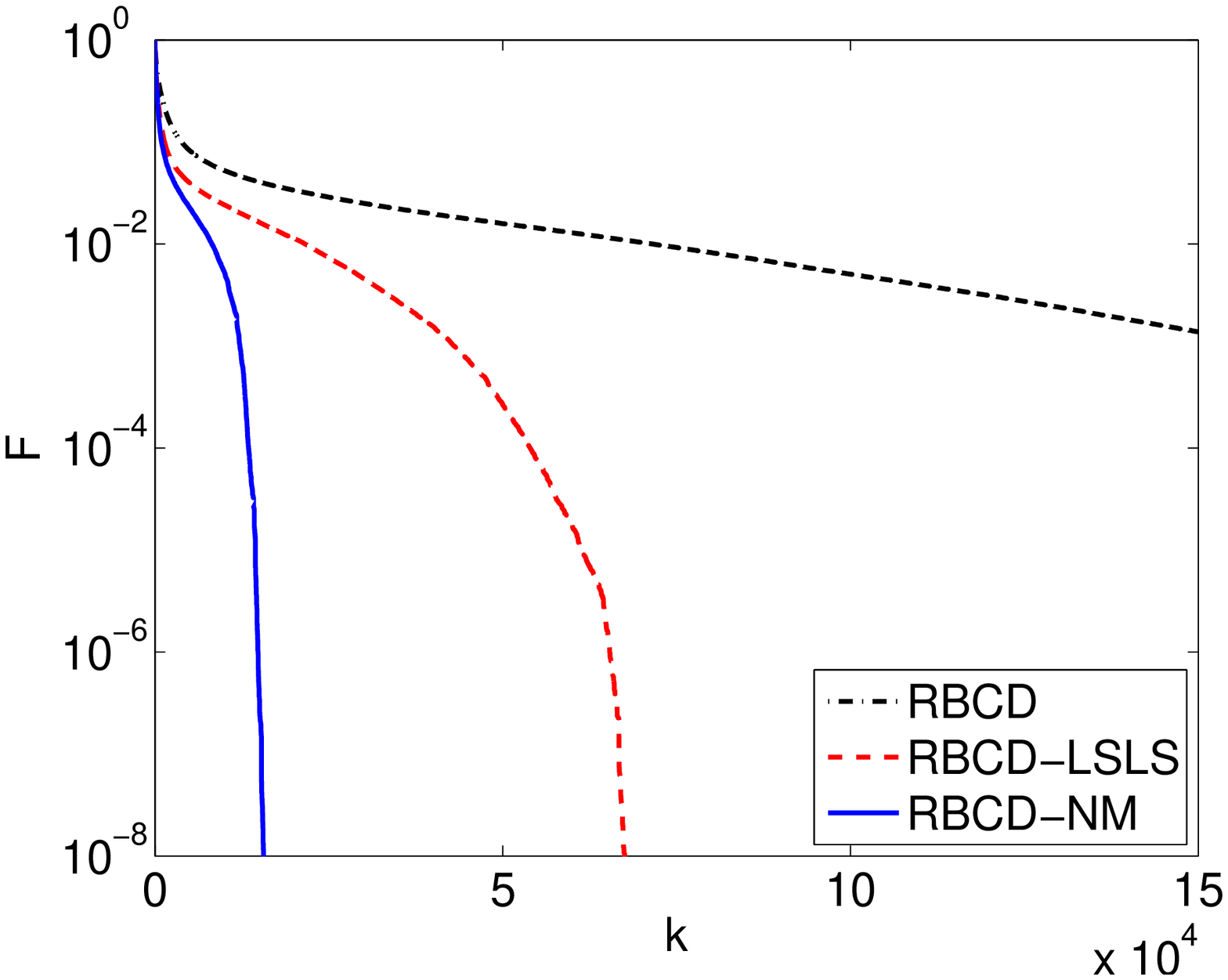} 
    \hspace{1ex}\mbox{}
    \label{fig:lasso-b20}}\\
    \subfloat[Blocksize $N_i=200$.]{
    \includegraphics[width=0.45\textwidth]{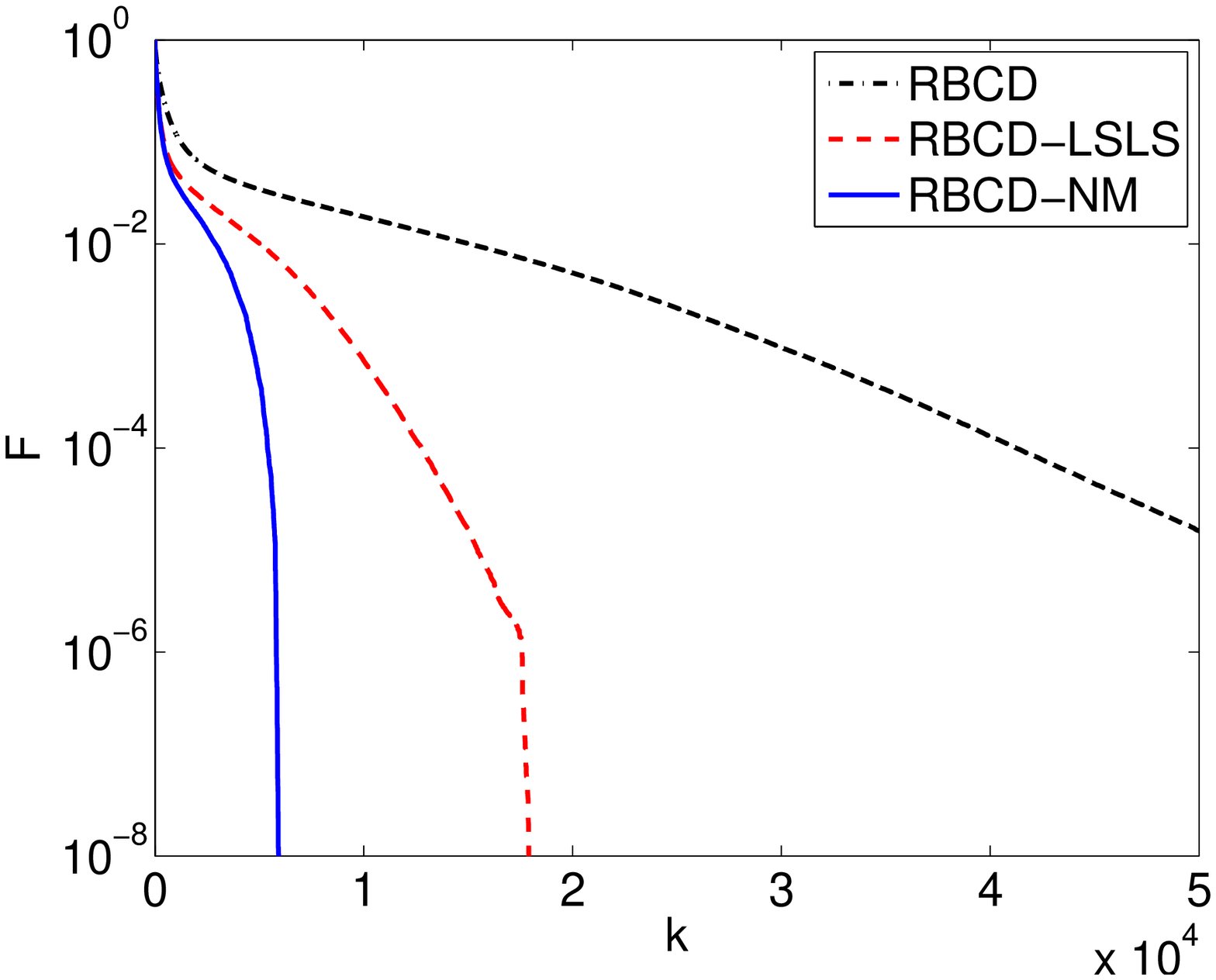}
    \label{fig:lasso-b200}}
    \hfill
    \subfloat[Blocksize $N_i=2000$.]{
    \includegraphics[width=0.47\textwidth]{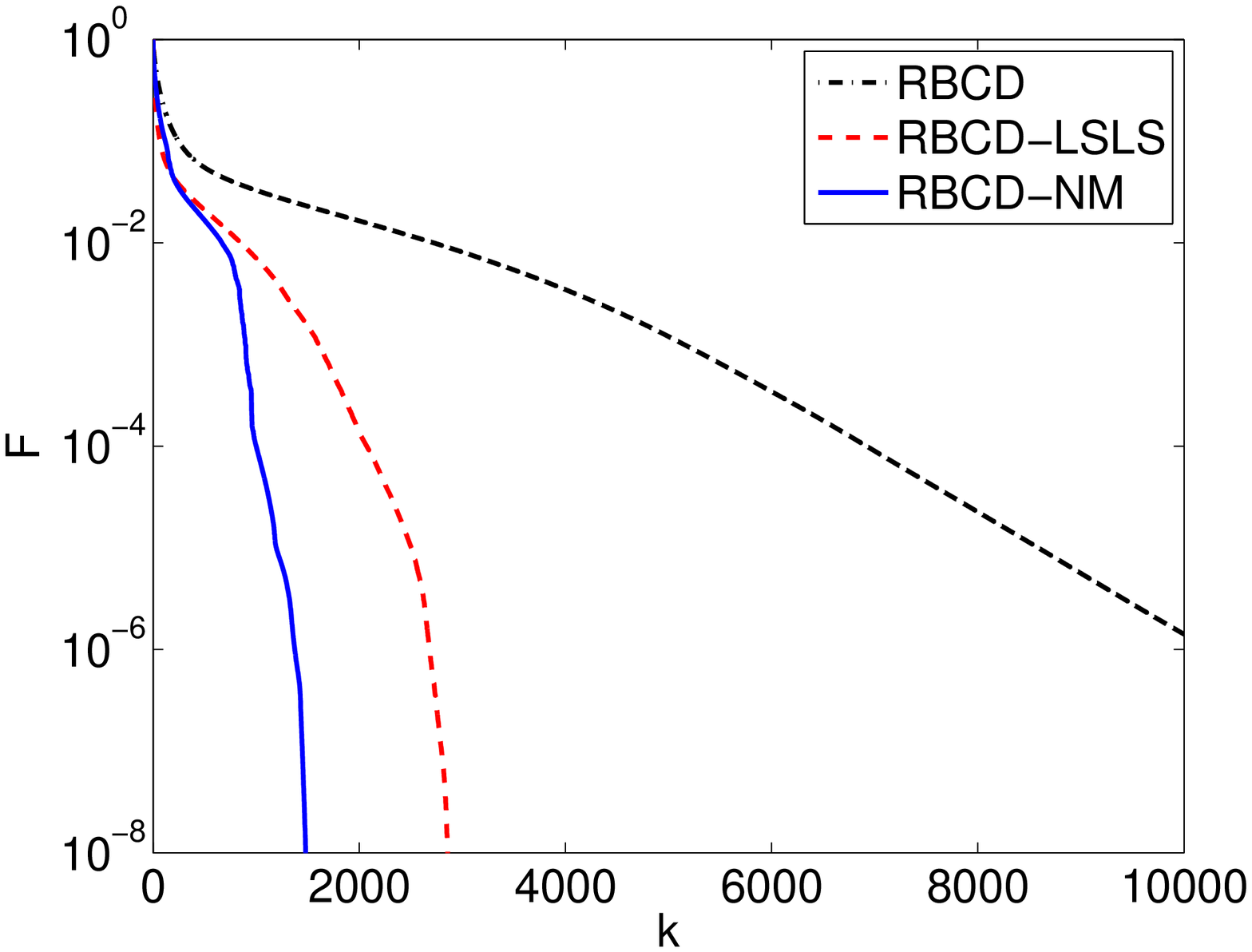}
    \label{fig:lasso-b2000}}
  \end{center}
  \caption{Comparison of different methods with block coordinate sizes
  $N_i=1, 20, 200, 2000$.}
  \label{fig:lasso}
\end{figure}

We run each algorithm with four different block coordinate sizes
$N_i=1, 20, 200, 2000$ for all $i$. 
For each blocksize, we pick the block coordinates uniformly at random
at each iteration.
Note that $N_i=2000=N$ gives the full gradient versions of the methods
considered, which are deterministic algorithms.
We choose the same initial point $x^0 = 0$ for all three methods.

For the RNBPG method, we used the parameters $M=10$, $\eta = 1.1$, $\underline\aalpha = 10^{-8}$, 
$\bar\aalpha = 10^8$ and $\sigma=10^{-4}$. In addition, we used the Barzilai-Borwein 
spectral method \cite{BaBo88} to compute the initial estimate $\theta_k^0$. 
That is, we choose
\[
\theta_k^0 = \frac{\|A_{:,i_k}u\|_2^2}{\|u\|_2^2},
\]
where 
\[
u = \arg\min_s \left\{\nabla_{i_k} f(x^k)^T s + \frac{L_{i_k}}{2}
\|s\|^2 + \q_{i_k}(x^k_{i_k}+s)\right\}, \quad L_{i_k} = \|A_{:,i_k}\|_2^2.
\]

Figure~\ref{fig:lasso} shows the behavior of different algorithms with the four different block 
coordinate sizes. For $N_i=1$ in Figure~\ref{fig:lasso}\subref{fig:lasso-b1}, RBCD has slightly better 
convergence speed than RBCD-LS and RNBPG.  The reason is that in this case, along each block $f$ 
becomes an one-dimensional quadratic function, and the value $L_i=\|A_{:,i}\|_2^2$ 
gives the accurate second partial derivative of $f$ along each dimension. 
Therefore in this case the RBCD method essentially uses the best step size, which is 
generally better than the ones used in RBCD-LS and RNBPG. 

When the blocksize $N_i$ is larger than one, the value $L_i=\|A_{:,i}\|_2^2$ 
is the magnitude of second derivative along the most curved direction.
Line search based methods may take advantage of the possibly much smaller
local curvature along the search direction by taking larger step sizes.
Figure~\ref{fig:lasso}~\subref{fig:lasso-b20}, \subref{fig:lasso-b200}
and~\subref{fig:lasso-b2000} show that RBCD-LS converges much faster that
RBCD while RNBPG (with $M=10$) converges substantially faster than RBCD-LS.


\begin{figure}[t]
  \psfrag{blk}[cc]{\small blocksize $N_i$}
  \psfrag{itr}[bc]{\small \# iterations}
  \psfrag{epo}[bc]{\small \# epochs}
  \psfrag{time}[bc]{\small Time}
  \psfrag{RBCD}[bl]{\footnotesize RBCD}
  \psfrag{RACD}[bl]{\footnotesize RACD}
  \psfrag{RBCD-LSLS}[bl]{\footnotesize RBCD-LS}
  \psfrag{RBCD-NM}[bl]{\footnotesize RNBPG}
  \begin{center}
    \subfloat[Number of iterations versus blocksize.]{
    \includegraphics[width=0.45\textwidth]{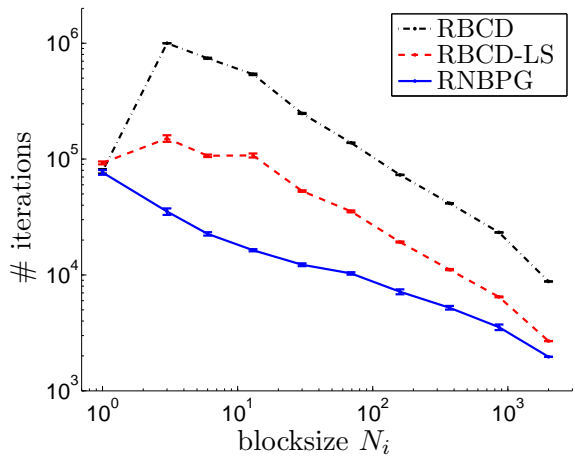}
    \label{fig:lasso-iter}}
    \hfill
    \subfloat[Number of epochs versus blocksize.]{
    \includegraphics[width=0.45\textwidth]{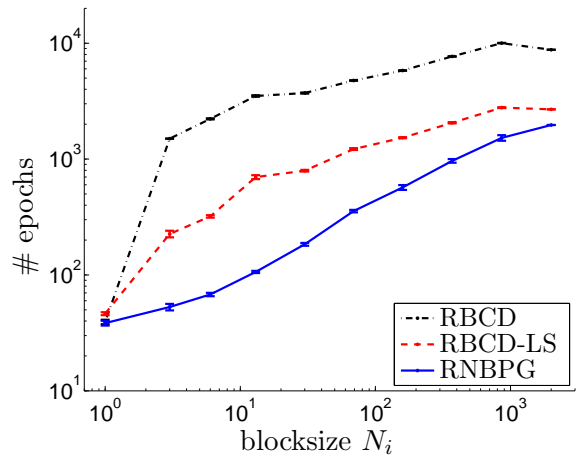} 
    \hspace{1ex}\mbox{}
    \label{fig:lasso-epch}}\\
    \subfloat[Computation time versus blocksize.]{
    \includegraphics[width=0.45\textwidth]{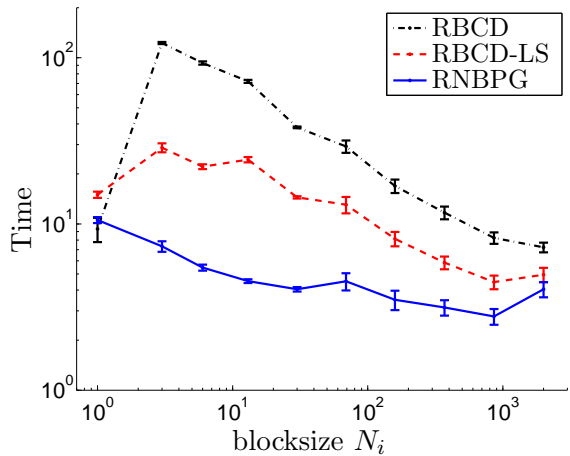}
    \label{fig:lasso-time}}
    \hfill
    \subfloat[Computation time versus blocksize for a different problem instance with $m=5000$ and $N=1000$.]{
    \includegraphics[width=0.47\textwidth]{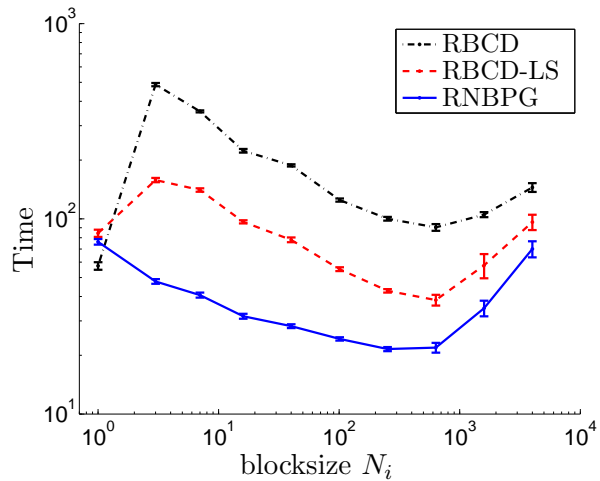}
    \label{fig:lasso-time-a}}
  \end{center}
  \caption{Comparison of different methods when varying the block coordinate size $N_i$.}
  \label{fig:lasso-varyblock}
\end{figure}

Figure~\ref{fig:lasso-varyblock} shows more comprehensive study
of the performance of the three methods: RBCD, RBCD-LS, and RNBPG.
Figure~\ref{fig:lasso-varyblock}\subref{fig:lasso-iter} 
shows the number of iterations of different methods required to reach 
the precision $F(x^k)-F(x^*)\leq 10^{-6}$, when using~10 different block sizes 
ranging from~1 to~2000 with equal logarithmic spacing.
Figure~\ref{fig:lasso-varyblock}\subref{fig:lasso-epch} 
shows the number of epochs required to reach the same precision,
where each epoch corresponds to one equivalent pass over the dataset 
$A\in\Re^{m\times N}$, that is, equivalent to $N/N_i$ iterations.
For each method and each block size, we record the results of 10 runs with
different random sequences to pick the block coordinates, and plot the mean
with the standard deviation as error bars. 
As we can see, the number of iterations in general decreases when we increase 
the block size, because each iteration involves more coordinates and more computation.
On the other hand, the number of epochs required increases with the block size,
meaning that larger block size updates are less efficient than small block size
updates. 

The above observations suggest that using larger block sizes is less efficient
in terms of the overall computation work (e.g., measured in total flops). 
However, this does not mean longer computation time.
In particular, using larger block sizes may better take advantage of 
modern multi-core computers for parallel computing, thus may take less
computation time.
Figure~\ref{fig:lasso-varyblock}\subref{fig:lasso-time} 
shows the computation time required to reach the same precision
on a 12 core Intel Xeon computer. 
We used the Intel Math Kernel Library (MKL) to carry out parallel
dense matrix and vector operations.
The results suggest that using appropriate large block size may take 
the least amount of computation time.
We note that such timing results heavily depend on the specific architecture
of the computer, in particular its cache size for fast access, 
the relative size of the data matrix~$A$, and other implementation details. 
For example, 
Figure~\ref{fig:lasso-varyblock}\subref{fig:lasso-time-a} 
shows the timing results on the same computer for a different problem instance 
with $m=2000$ and $N=4000$. 
Here the best block size is smaller than one shown in 
Figure~\ref{fig:lasso-varyblock}\subref{fig:lasso-time}, because
the size of each column of~$A$ is doubled and 
the operations involved in each coordinate 
(corresponding to a column of the matrix) has increased.
However, for any fixed block size, 
the relative performance of the three algorithms are consistent;
in particular, RNBPG substantially outperforms the other two methods
in most cases.

\begin{figure}[t]
  \psfrag{k}[cc]{\small Iteration number $k$}
  \psfrag{F}[bc]{\small $F(x^k)-F^*$}
  \psfrag{RBCD}[bl]{\footnotesize RBCD}
  \psfrag{RACD}[bl]{\footnotesize RACD}
  \psfrag{RBCD-LSLS}[bl]{\footnotesize RBCD-LS}
  \psfrag{RBCD-NM}[bl]{\footnotesize RNBPG}
  \begin{center}
    \subfloat[RCV1 dataset with blocksize $N_i=100$.]{
    \includegraphics[width=0.45\textwidth]{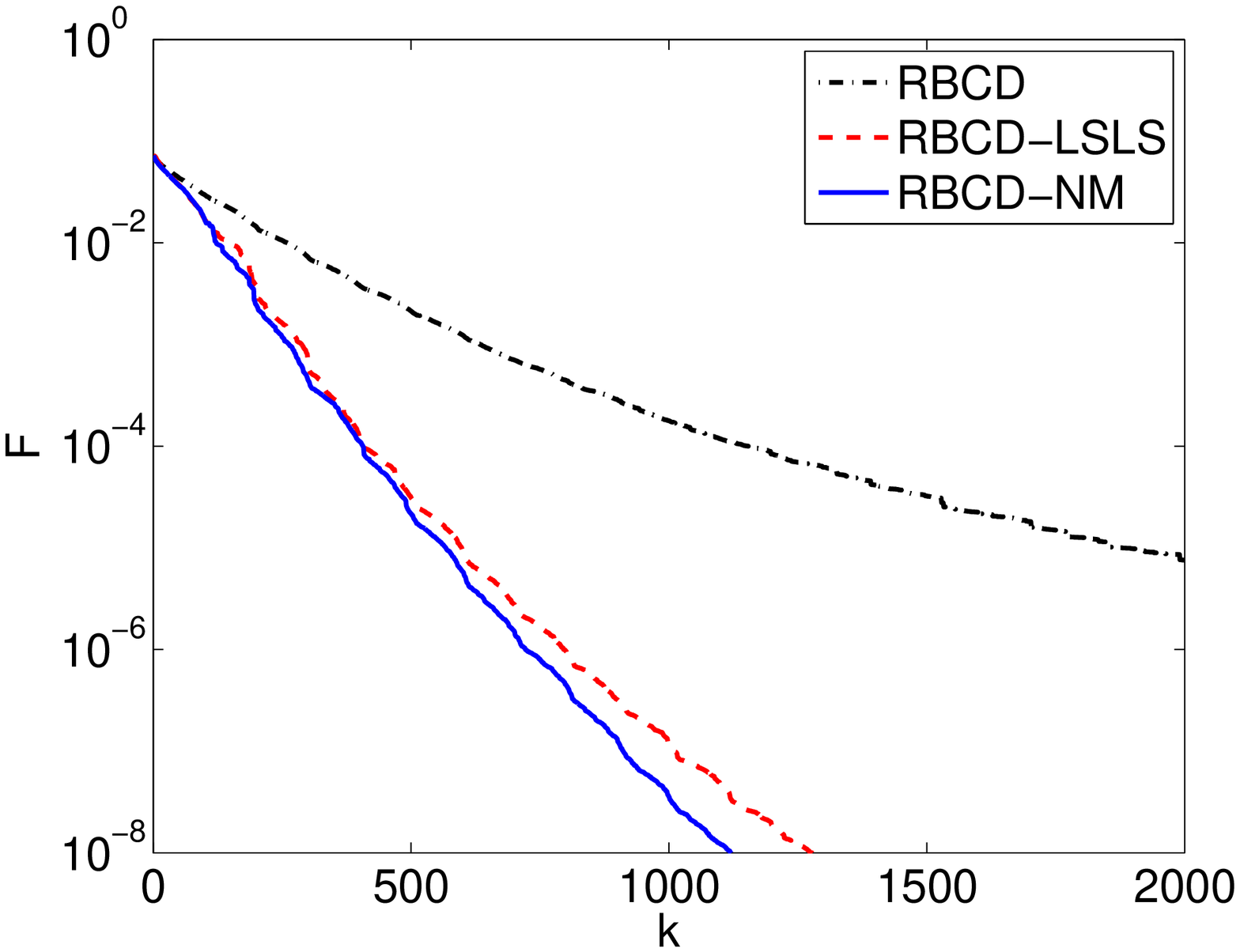}
    \label{fig:rcv1b100}}
    \hfill
    \subfloat[RCV1 dataset with blocksize $N_i=1000$.]{
    \includegraphics[width=0.45\textwidth]{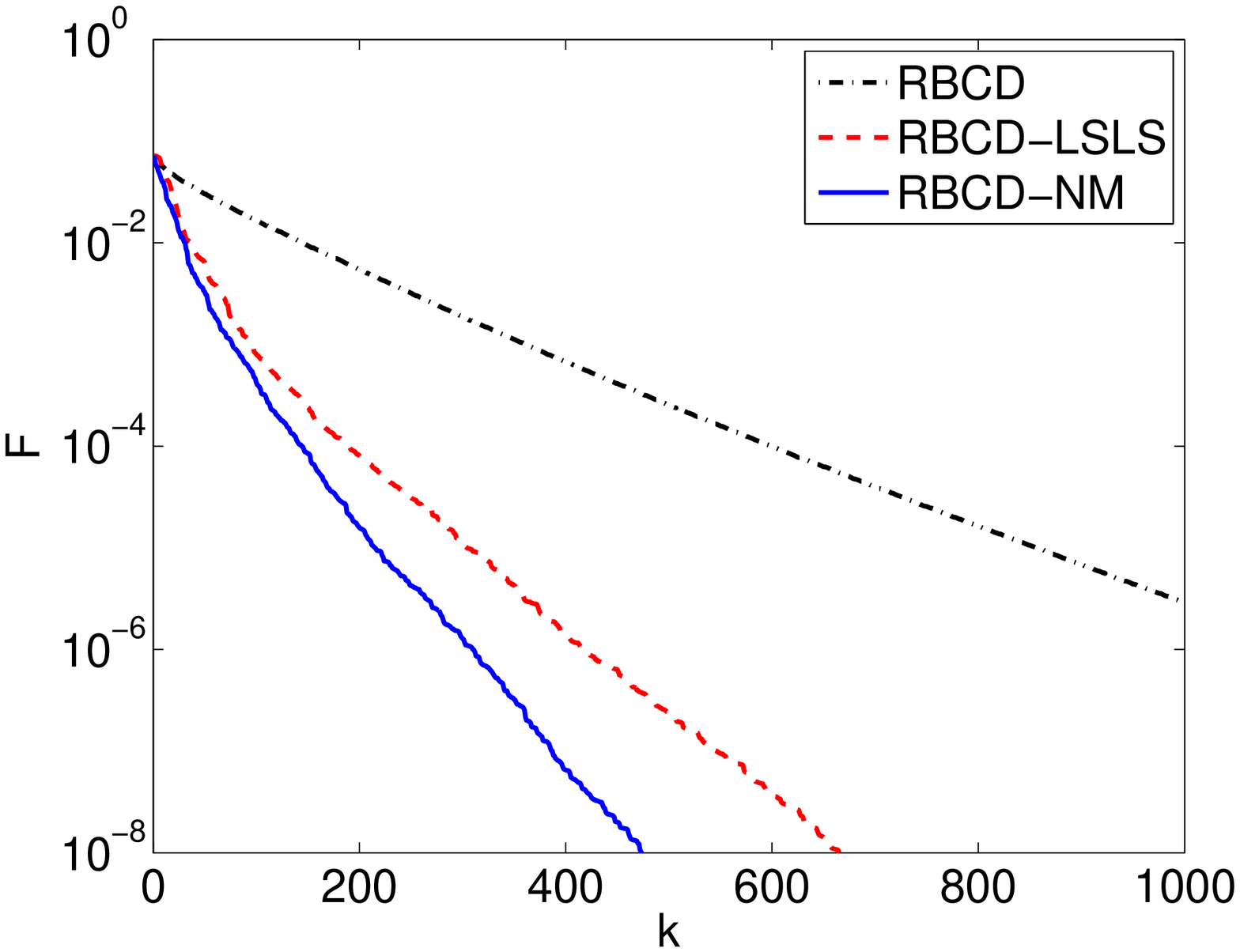} 
    \hspace{1ex}\mbox{}
    \label{fig:rcv1b1000}}\\
    \subfloat[News20 dataset with blocksize $N_i=100$.]{
    \includegraphics[width=0.45\textwidth]{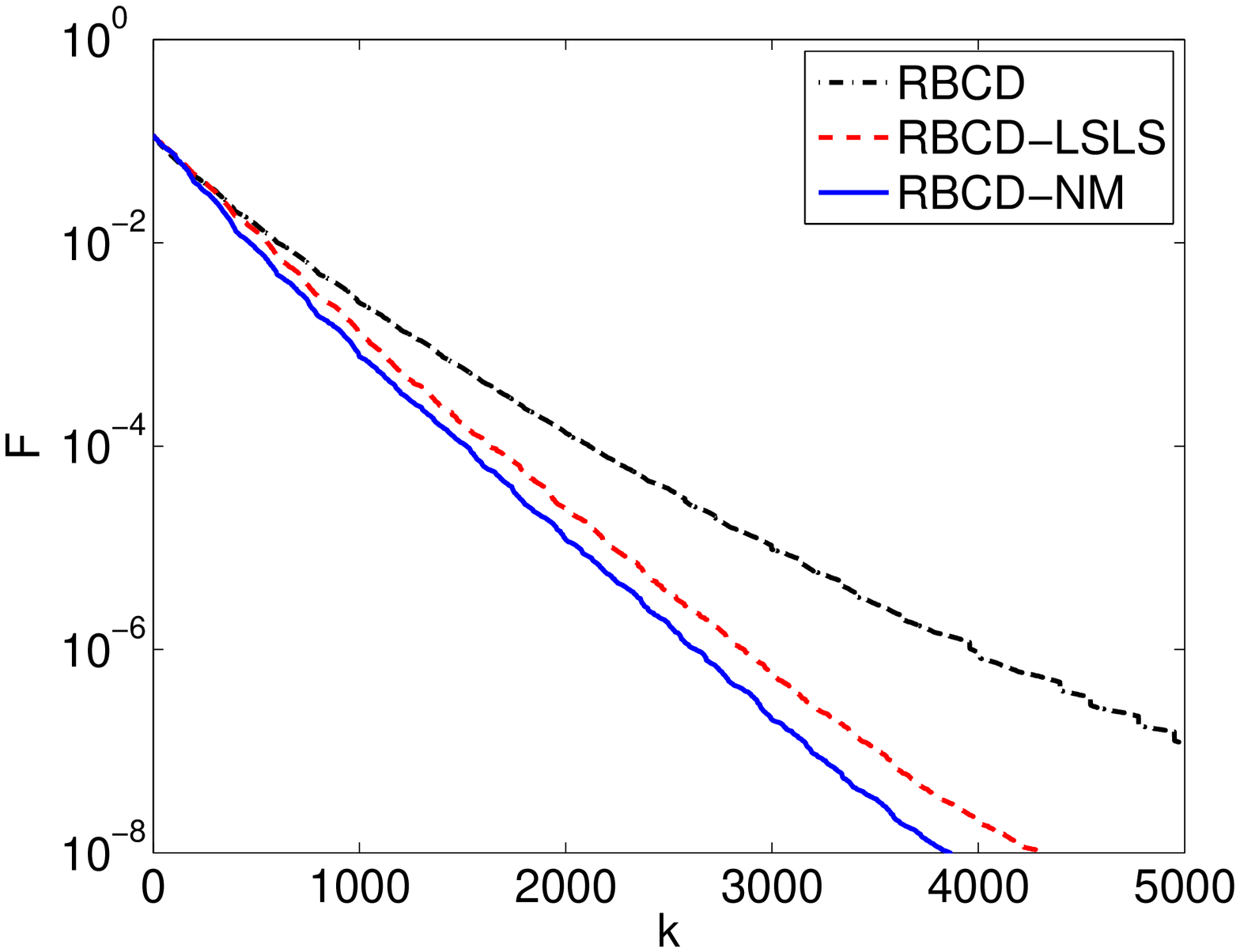}
    \label{fig:news20b100}}
    \hfill
    \subfloat[News20 dataset with blocksize $N_i=1000$.]{
    \includegraphics[width=0.47\textwidth]{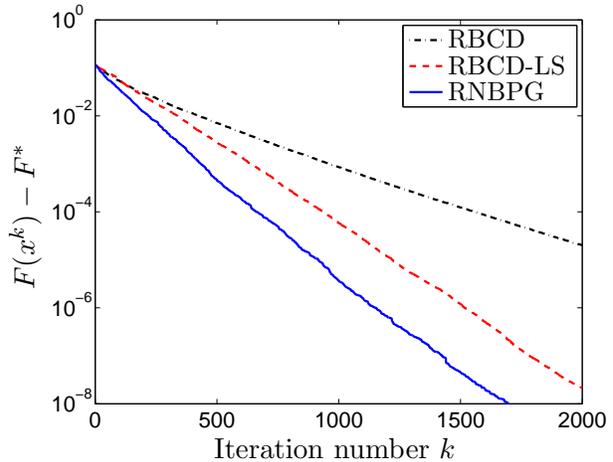}
    \label{fig:news20b1000}}
  \end{center}
  \caption{Comparison on the dual empirical risk minimization problem with real datasets.}
  \label{fig:dualerm}
\end{figure}

We also conducted experiments on using randomized block coordinate methods
to solve a dual SVM problem in machine learning 
(specifically, the dual of a smoothed SVM problem
described in \cite[Section~6.2]{ShZh13}).
We used two real datasets from the LIBSVM web site \cite{LIBSVMdata},
whose characteristics are summarized in Table~\ref{tab:datasets}.
In the dual SVM problem, the dimension of the dual variables are the same
as the number of samples~$N$, and we  partition the dual variables into
blocks to apply the three randomized block coordinate gradient methods.
Figure~\ref{fig:dualerm} shows the reduction of the objective value
with the three methods on the two datasets, each illustrated with
two block sizes: $N_i=100$ and $N_i=1000$.
We observe that the RNBPG method converges faster than the other two methods, 
especially with relatively larger block sizes.

\begin{table}[t]
    \centering
    \begin{tabular}{|l|r|r|r|r|}
        \hline
     & number of samples $N$ & number of features $d$ & sparsity & $\lambda$\\
        \hline
        \texttt{RCV1}    & 20,242 &  47,236 & 0.16\% & 0.0001 \\
        \texttt{News20}  & 19,996 &  1,355,191 & 0.04\%  & 0.0001\\
        \hline
    \end{tabular}
    \caption{Characteristics of two sparse datasets from the LIBSVM web site \cite{LIBSVMdata}.}
    \label{tab:datasets}
\end{table}

To conclude, our experiments on both synthetic and real datasets 
clearly demonstrate the advantage of the nonmonotone 
line search strategy (with spectral initialization) 
for randomized block coordinate gradient methods.

\end{document}